\documentclass[3p]{elsarticle}

\usepackage{amsmath}
\usepackage{amsfonts}
\usepackage{paralist}
\usepackage{booktabs}

\usepackage{paralist}
\usepackage{url}
\usepackage{hyperref}
\usepackage{color}
\usepackage{framed}
\usepackage{graphicx}
\usepackage{mathrsfs}
\usepackage[T1]{fontenc}

\newtheorem{theorem}{Theorem}[section]
\newtheorem{prop}[theorem]{Proposition}
\newtheorem{cor}[theorem]{Corollary}
\newtheorem{lemma}[theorem]{Lemma}
\newtheorem{algorithm}[theorem]{Algorithm}
\newtheorem{remark}[theorem]{Remark}
\newtheorem{define}[theorem]{Definition}
\newtheorem{example}[theorem]{Example}

\newtheorem{condition}[theorem]{Condition}

\def\citet{\cite}
\newcommand{\comment}[1]{}
\def\cal{\mathcal}

\newcommand{\rank}{{\sf rank}}
\newcommand{\jac}{{\sf Jac}}

\newcommand{\RR}{{\mathbb R}}
\newcommand{\NN}{{\mathbb N}}
\newcommand{\CC}{{\mathbb C}}

\newcommand{\mylinebreak}{\hfill\par\vspace{1.0ex}

\def\ow{o\kern-.42em\raise.82ex\hbox{
  \vrule width .12em height .0ex depth .075ex \kern-0.16em \char'56}\kern-.07em}
\def\OW{O\kern-.460em\raise1.36ex\hbox{
\vrule width .13em height .0ex depth .075ex \kern-0.16em \char'56}\kern-.07em}

\noindent\hspace*{1em}}

\newcommand{\crit}{{\sf Crit}}

\newcommand{\sing}{{\sf Sing}}
\newcommand{\bfz}{\mathbf{0}}
\newcommand{\bfV}{\mathbf{V}}
\newcommand{\bfB}{\mathbf{B}}

\newcommand{\SSn}{\mathcal{S}^{n\times n}_{++}}
\newcommand{\Sn}{\mathcal{S}^{n\times n}}
\newcommand{\ud}{\mbox{\upshape d}}
\newcommand{\td}{\tilde}
\newcommand{\wt}{\widetilde}

\numberwithin{equation}{section}

\newenvironment{proof}{\text{\it Proof. }}{\hfill $\square$\par\hfill}
\makeatletter
\def\ps@pprintTitle{%
 \let\@oddhead\@empty
 \let\@evenhead\@empty
 \def\@oddfoot{\centerline{\thepage}}%
 \let\@evenfoot\@oddfoot}
\makeatother

\begin{document}

\begin{frontmatter}
\title{On types of degenerate critical points of real polynomial functions\tnoteref{t1}}
\tnotetext[t1]{The first author was supported by
the Chinese National Natural Science Foundation under grants 11401074, 11571350.
The second author was partially supported by Vietnam National Foundation for Science and Technology Development (NAFOSTED), grant 101.04-2016.05.}

\author[1]{Feng Guo}
\ead{fguo@dlut.edu.cn}

\author[2]{Ti{\' {\^ e}}n-So\kern-.42em\raise.82ex\hbox{ \vrule width .12em height .0ex depth .075ex \kern-0.16em \char'56}\kern-.07em n Ph\d{a}m}
\ead{sonpt@dlu.edu.vn}

\address[1]{School of Mathematical Sciences,\\
Dalian  University of Technology,
Dalian, 116024, China}

\address[2]{Department of Mathematics, \\
University of Dalat,
1 Phu Dong Thien Vuong, Dalat, Viet Nam}

\begin{abstract}
	In this paper, we consider the problem of identifying the type
	(local minimizer, maximizer or saddle point) of a given
	isolated real critical point $c$, which is degenerate,
	of a multivariate polynomial function $f$.
	To this end, we introduce the definition of faithful radius of $c$
	by means of the curve of tangency of $f$. We show that
	the type of $c$ can be determined by the global extrema
	of $f$ over the Euclidean ball centered at $c$ with a faithful radius.
	%Assuming that an isolation radius of $c$ is known,
	We propose algorithms to compute a faithful radius of
	$c$ and determine its type.
\end{abstract}


\date{}

\end{frontmatter}

%\title[degenerate critical points of polynomial functions]
%{On degenerate critical points of real polynomial functions}
%\author{Feng Guo}
%\address{School of Mathematical Sciences, Dalian  University of Technology,
%Dalian, 116024, China}
%\email{fguo@dlut.edu.cn}
%\author{ Ti{\' {\^ e}}n-S{\' o}n Ph\d{a}m}
%\address{Department of Mathematics, University of Dalat, 1 Phu Dong Thien Vuong, Dalat, Viet Nam}
%\email{sonpt@dlu.edu.vn}
%
%
%
%
%\begin{abstract}
%	In this paper, we consider the problem of identifying the type
%	(local minimizer, maximizer or saddle point) of a given
%	isolated real critical point $c$, which is degenerate,
%	of a multivariate polynomial function $f$.
%	To this end, we introduce the definition of faithful radius of $c$
%	by means of the curve of tangency of $f$.
%	The type of $c$ can be determined by the local optima
%	of $f$ over the Euclidean ball centered at $c$ with a faithful radius.
%	Assuming that an isolation radius of $c$ is known,
%	we propose algorithms and show that how to compute a faithful radius of
%	$c$ and determine its type.
%\end{abstract}
%
%
%\maketitle

%\input{algorithm}
\section{Introduction}
Let $f\in\RR[X]:=\RR[X_1,\ldots,X_n]$, the polynomial ring over the real number field
$\RR$ with $n$ variables.
Throughout this paper, we denote upper case letters (like $X, Y$) as variables
and lower case letters (like $x, y$) as points in the ambient spaces.
Denote $\bfz$ as the origin or the vector of zeros.
Given $c\in\RR^n$ such that the gradient
$\nabla f(c)=\bfz$ and the Hessian matrix
$\nabla^2 f(c)$ is singular, i.e. $c$ is a degenerate real critical
point of $f$.
An interesting problem is to identify the type of $c$, i.e. is $c$ a local minimizer,
maximizer or saddle point of $f$?
To solve it, it is intuitive to consider the higher order partial derivatives of $f$ at $c$.
However, to the best of our knowledge, it is difficult to obtain a straightforward and simple
method, which takes into account the higher order derivatives of $f$,
to systematically solve this problem. When $f$ is a sufficiently smooth function
(not necessarily a polynomial), some partial answers to this problem were given in
\cite{Bolis1980,Cushing1975} under certain assumptions on its Taylor expansion at $c$.
When $f$ is a multivariate real polynomial, Qi investigated
its critical points and extrema structures in \cite{Qi2004} without giving a computable
method to determine their types. Nie gave a numerical method in \cite{Nie2015} to compute
all $H$-minimizers (critical points at which the Hessian matrices are positive semidefinite)
of a polynomial by semidefinite relaxations. However, there is no completed procedure in
\cite{Nie2015} to verified that a $H$-minimizer is a saddle point.

Without loss of generality, we
suppose that $c=\bfz$ and $f(\bfz)=0$.
In this paper, we consider the case when
\begin{center}
{\itshape $\bfz$ is an {\bf\itshape isolated} real critical point of $f$},
\end{center}
i.e., there exists a
neighborhood $\mathcal{O}\subseteq\RR^n$ of $\bfz$ such that $\bfz$ is the only real
critical point of $f$ in $\mathcal{O}$. It is well known that any small changes
of the coefficients of $f$ may render $\bfz$ nondegenerate.
Hence, we aim to present a computable and symbolic method
to determine the type of $\bfz$.
%To the best of our knowledge, there is little related work
%in a symbolic way.

%\paragraph{\bf Methodology}

Now let us briefly introduce the basic idea we use
to deal with this problem and the contribution made in this paper.
Denote $\RR_+$ as the set of positive real numbers and $\Vert x\Vert_2$ as the
Euclidean norm of $x\in\RR^n$. For any $r\in\RR_+$,
let
\begin{equation}\label{eq::BS}
	\mathsf{B}_r:=\{x\in\RR^n\mid \Vert x\Vert_2\le r\}\quad
	\text{and}\quad\mathsf{S}_r:=\{x\in\RR^n\mid \Vert x\Vert_2 = r\}.
\end{equation}
%Denote $\mathsf{B}_r\subseteq\RR^n$ and $\mathsf{S}_r$ as the ball centered at the origin with
%radius $R$.
Define
\begin{equation}\label{eq::ops}
	f_r^{\min}:=\min\{f(x)\mid x\in \mathsf{B}_r\}\quad\text{and}\quad
	f_r^{\max}:=\max\{f(x)\mid x\in \mathsf{B}_r\}.
\end{equation}
Obviously, it holds that
%the type of critical point $\bfz$ determines the local extrema of
%$f$ over $\mathsf{B}_r$, that is
\begin{enumerate}[\upshape (i)]
	\item\label{case1} if $\bfz$ is a local minimizer, then $f_r^{\max}>0$ and $f_r^{\min}=0$
for some $r\in\RR_+$;
\item\label{case2} if $\bfz$ is a local maximizer, then $f_r^{\max}=0$ and $f_r^{\min}<0$
for some $r\in\RR_+$;
\item\label{case3} if $\bfz$ is a saddle point, then $f_r^{\max}>0$ and $f_r^{\min}<0$
for any $r\in\RR_+$.
\end{enumerate}
Now we consider the above statements the other way around.
That is, can we classify the degenerate critical point
$\bfz$ by the signs of
%local optima
$f_r^{\min}$ and $f_r^{\max}$? Two issues have to be addressed.
%of $f$ over some ball $\mathsf{B}_r$?
\begin{enumerate}[\upshape (1)]
\item If $\bfz$ is a local minimizer or maximizer, it can be certified by
giving a radius $r$ such that
$f_r^{\min}=0$ or $f_r^{\max}=0$. The difficulty is that if $f_r^{\min}<0<f_r^{\max}$
for some $r\in\RR_+$, then what is the type of $\bfz$?
Since we do not know if the radius $r$ is sufficiently small, we can not
claim that $\bfz$ is a saddle point. For example, consider the polynomial
$f(X_1,X_2)=X_1^2+(1-X_1)X_2^4$ (Example \ref{ex::ex1})
with $\bfz$ being an isolated real critical point.
Notice that $\bfz$ is degenerate.
If we choose $r=2\sqrt{2}$, then
$f_r^{\min}\le f(2,2)=-12<0<52=f(-2,2)\le f_r^{\max}$. However,
if $r<1$, we have $f(x_1,x_2)>0$ for any
$x\in\mathsf{B}_r\backslash\{\bfz\}$ and $f_r^{\min}=0$.
Hence, $\bfz$ is not a saddle point but a strict local minimizer.
Therefore, in order to determine the type of $\bfz$ by
%local optima
the global extrema of $f$ over some ball $\mathsf{B}_r$,
we need to ensure that the radius $r$ is sufficiently small.
\item
Note that the optimization problems in $(\ref{eq::ops})$ are
themselves NP-hard. In particular, maximizing a cubic polynomial over a unit ball
is NP-hard \cite{RW}. Numerically, approximation methods for polynomial
optimization problems based on semidefinite relaxations have been extensively studied
\cite{HPh, Lasserre09, Laurent_sumsof, PPSOSMarshall}.
However, since we need to certify that $f_r^{\min}$ and $f_r^{\max}$
in cases (\ref{case1}) and (\ref{case2}) are exact $0$, any numerical errors in the output
of approximation methods for $(\ref{eq::ops})$ may mislead us to the
wrong case (\ref{case3}). Symbolically, a univariate polynomial whose roots contain
$f_r^{\min}$ and $f_r^{\max}$ may be obtained by means of the KKT system of (\ref{eq::ops})
and some elimination computations as in \cite{XZ2015}.
However, to determine the signs of $f_r^{\min}$
and $f_r^{\max}$, extra symbolic computations are needed to find one point in some real
algebraic set or to certify its emptiness \cite{Safey08, SafeyElDin2003}.
\end{enumerate}

In this paper, we aim to tackle the above issues and classify isolated
critical point $\bfz$ of $f$ by its global extrema over some Euclidean balls.
For the first issue,
we define the so-called {\itshape faithful radius} (Definition \ref{def::faithr})
of $\bfz$,
via the curve of tangency of $f$ at $\bfz$ which is studied in \cite{Durfee,VS},
such that
%the implications in the other direction in the statements (\ref{case1}),
%(\ref{case2}) and (\ref{case3}) also hold
the type of $\bfz$ can be determined by the signs of $f^{\max}_r$ and
$f^{\min}_r$ for any faithful radius $r$ of $\bfz$ (Theorem \ref{th::main}).
Provided that an isolation radius (Definition \ref{def::is}) of $\bfz$ is known,
we propose an algorithm (Algorithm \ref{al::fr}) to compute a faithful radius of $\bfz$.
We also discuss some strategies to compute an isolation radius of $\bfz$.
For the second issue, instead of computing the extrema
$f_r^{\max}$ and $f_r^{\min}$ in $(\ref{eq::ops})$, we present an algorithm (Algorithm \ref{al::main})
to identify the type of $\bfz$ by
computing isolating intervals for each real root of a zero-dimensional polynomial system,
which can be done by, for example, the Rational Univariate Representations (RUR) \cite{Rouillier1999}
for multivariate polynomial systems.
%Let $f_*$ and $f^*$ be reached at $v$,
%and $u\in B(c,R)$, respectively.
%If
%\begin{enumerate}[(1)]
%\item There are two pieces of differentiable curves $\mathcal{C}_v,
%\mathcal{C}_u$ connecting $c$ and $v$, $c$ and $u$, repectively;
%\item For any nonzero point $x$ in $\mathcal{C}_v, \mathcal{C}_u$, we have
%$f(x)\neq 0$,
%\end{enumerate}
%then by the continuity and the mean value theorem, we can claim that $c$ is a
%saddle point.

\vskip 5pt
The paper is organized as follows.
Some notation and preliminaries used in this paper are given in Section \ref{sec::pre}.
We define the faithful radius of $\bfz$ in Section \ref{sec::faithfulradius}.
An algorithm for computing a faithful radius of $\bfz$ is presented in Section
\ref{sec::cmptr}. We show that how to determine the type of $\bfz$ in a symbolic way
in Section \ref{sec::types}. Some conclusion is made in Section \ref{sec::cons}.

\section{Notation and preliminaries}\label{sec::pre}
The symbol $\RR$ (resp., $\CC$) denotes the set of real (resp., complex) numbers.
Denote $\RR^{n\times n}$ (resp., $\CC^{n\times n}$) as the set of $n\times n$ matrices
with real (resp. complex) number entries.
$\RR[X]=\RR[X_1,\ldots,X_n]$ denotes the ring of polynomials in variables
$X=(X_1,\ldots,X_n)$ with real coefficients. Denote $\Vert X\Vert^2_2\in\RR[X]$ as the polynomial
$X_1^2+\cdots+X_n^2$ in variables $X$
while $\Vert x\Vert_2$ as the Euclidean norm of $x\in\RR^n$.
If $f, g$ are two functions with suitably chosen domains and codomains, then
$f\circ g$ denotes the composite function of $f$ and $g$.

A subset $I\subseteq\RR[X]$ is called an ideal if $0\in I$, $I+I\subseteq I$ and
$p\cdot q\in I$ for all $p\in I$ and $q\in\RR[X]$.
The product of  two ideals $I$ and $J$ in $\mathbb{R}[X],$
denoted by $I \cdot J,$ is the ideal generated by all products $f \cdot g$ where $f \in I$ and $g \in J.$
For $g_1,\ldots,g_s\in\RR[X]$,
denote $\langle g_1,\ldots,g_s\rangle$ as the ideal in $\RR[X]$ generated by $g_i$'s,
i.e., the set $g_1\RR[X]+\cdots+g_s\RR[X]$.
An ideal is radical if $f^m\in I$ for some integer $m\ge 1$ implies that $f\in I$.
The radical of an ideal $I\subseteq\RR[X]$, denoted $\sqrt{I}$, is the set
$\{f\in\RR[X]\mid f^m\in I\text{ for some integer }m\ge 1\}$.
An (resp. real) affine variety is a subset of $\CC^n$ (resp. $\RR^n$) that consists of common
zeros of a set of polynomials.
For an ideal $I\subseteq\RR[X]$, denote $\bfV_\CC(I)$ and $\bfV_\RR(I)$ as the affine
varieties defined by $I$ in $\CC^n$ and $\RR^n$, respectively.
For a polynomial $g\in\RR[X]$,
respectively replace $\bfV_\CC(\langle g\rangle)$ and $\bfV_\RR(\langle g\rangle)$ by
$\bfV_\CC(g)$ and $\bfV_\RR(g)$ for simplicity.
Given a set
$V\subseteq\CC^n$, denote $\mathbf{I}(V)\subseteq\RR[X]$ as the vanishing ideal of
$V$ in $\RR[X]$, i.e., the set of all polynomials in $\RR[X]$ which equal zero at
every point in $V$.
For an ideal $I\subseteq\RR[X]$, denote $\dim(I)$ as the Hilbert
dimension of $I$, i.e., the degree of the affine Hilbert polynomial of $I$.
For an ideal $I\subseteq\RR[X]$,
the decomposition $I=I_1\cap\cdots\cap I_s$ is called the equidimensional decomposition
of $I$ if each ideal $I_i$ is pure dimensional, i.e., all its associated
primes have the same dimension.
For an affine variety $V\subseteq\CC^n$,
denote $\dim(V)=\dim(\mathbf{I}(V))$ as its dimension.
When $\bfV_\CC(I)$ is finite, the ideal $I$ is called to be zero-dimensional.
For any subset $S\subseteq\CC^n$, denote $\overline{S}^\mathcal{Z}$ as the Zariski
closure of $S$ in $\CC^n$, i.e., $\overline{S}^{\mathcal{Z}}=\bfV_\CC(\mathbf{I}(S))$.
The $l$-th elimination ideal $I_l$ of an ideal $I\in\RR[X]$ is the ideal of
$\RR[X_{l+1},\ldots,X_n]$ defined by $I_l=I\cap\RR[X_{l+1},\ldots,X_n]$ which can be
computed by the Groebner basis of $I$ with respect to an elimination order of $X$.
For more basic concepts from algebraic geometry, we refer to \cite{CLO,singular}.
The following procedures are considered as black boxes
in this paper (c.f. \cite{ARG,singular}):
\begin{enumerate}[-]
	\item Compute the Hilbert dimension of a given ideal $I\subseteq\RR[X]$;
	\item Compute the equidimensional decomposition of a given ideal $I\subseteq\RR[X]$;
	\item Test whether a given ideal $I\subset\RR[X]$ is radical and compute
		$\sqrt{I}$ if it is not;
	\item Compute the vanishing ideal $\mathbf{I}
		\left(\overline{\bfV_\CC(I)\backslash\bfV_\CC(J)}^{\mathcal{Z}}\right)
		\subseteq\RR[X]$ for some ideals $I, J\subseteq\RR[X]$;
	\item Compute isolating intervals for each real root of a zero-dimensional
		polynomial system. This can be done by, for example, the
		Rational Univariate Representations (RUR) \cite{Rouillier1999}
		for multivariate polynomial systems.
\end{enumerate}

We recall some background in real algebraic geometry and refer to
\cite{realAG} for more details.
A semi-algebraic subset of $\RR^n$ is a subset of $\RR^n$ satisfying a
boolean combination of polynomial equations and inequalities with real coefficients.
In this paper, $\RR^n$ will always be considered with its Euclidean topology,
unless stated otherwise.
Let $S_1\subseteq\RR^m$ and $S_2\subseteq\RR^n$ be two semi-algebraic sets. A
mapping $\psi: S_1\rightarrow S_2$ is semi-algebraic if its graph is semi-algebraic
in $\RR^{m+n}$.
\begin{theorem}\cite[Theorem 2.3.6]{realAG}\label{th::decomp}
	Every semi-algebraic subset $S$ of $\RR^n$ is the disjoint union of a finite
	number of semi-algebraic sets $\cup_{i=1}^s S_i$.
	Each $S_i$ is semi-algebraically homeomorphic
	to an open hypercube $(0,1)^{d_i}\subseteq\RR^{d_i}$
	for some $d_i\in\NN$.
\end{theorem}
The dimension $\dim(S)$ of a semi-algebraic set $S\subseteq\RR^n$ is the maximum
of $d_i$ as in Theorem \ref{th::decomp}.
A subset $S\subseteq\RR^n$ is connected if for every pair of sets $S_1$
and $S_2$ closed in $S$, disjoint and satisfying $S_1\cup S_2=S$, one has
$S_1=S$ or $S_2=S$.
Given points $p_1$ and $p_2$ of a subset $S\subseteq\RR^n$, a path in $S$
from $p_1$ to $p_2$ is a continuous map $\varphi: [a, b]\rightarrow S$ of some
closed interval in the real line into $S$, such that $\varphi(a)=p_1$ and
$\varphi(b)=p_2$. A subset $S\subseteq\RR^n$ is said to be path connected if
every pair of points of $S$ can be jointed by a path in $S$.
%\subsection{Faithful Radii}

%[semi-algebraic set, topology, connected, path connected, curve selection lemma]
Combining Theorems 2.4.4, 2.4.5 and Proposition 2.5.13 in \cite{realAG},
it follows that
\begin{prop}\label{prop::connected}
Let $S$ be a semi-algebraic set of $\RR^n$.
Then,
\begin{enumerate}[\upshape (i)]
	\item $S$ has a finite number of connected components which are closed
		in $S$;
	\item $S$ is connected if and only if it is path connected.
\end{enumerate}
\end{prop}
Hence, the rest of this paper, by saying that a semi-algebraic subset of $\RR^n$
is connected, we also mean that it is path connected.

\begin{theorem}\cite[Curve selection lemma]{Milnor1968}
	Let $S$ be a semi-algebraic subset of $\RR^n$ and $x\in\RR^n$ a point
	belonging to the closure of $S$. Then there exists an analytic
	semi-algebraic mapping $\varphi: [0, \epsilon] \rightarrow \RR^n$ such that
$\varphi(0) = x$ and $\varphi((0, \epsilon])\subset S$.
\end{theorem}

%In this paper, we use the words ¡°generic¡± and ¡°genericity¡±
%as conditions on the input data for some property to hold, and
%they shall mean for all but a set of Lebesgue measure zero in
%the space of data.

\section{Faithful radius and types of degenerate critical points}\label{sec::faithfulradius}

In the rest of this paper, we always denote $f$ as the considered polynomial in
$\RR[X]$ with $\bfz$ being an isolated real critical point.

Denote $\crit_\RR(f)$ and $\crit_\CC(f)$ as the sets of real and complex
critical points of $f$, respectively.
%For any
%$0<r\in\RR$, let $\mathsf{B}_r$ as the ball centered at the origin with radius
%$r$.
Define
\begin{equation}\label{eq::gamma}
\Gamma_\RR(f):=\left\{x\in\RR^n\mid \exists \lambda\in\RR
	\quad\text{s.t.}\ \nabla f(x)=\lambda x\right\}.
%	\frac{\partial{f}}{\partial{x_i}}x_j-
%\frac{\partial{f}}{\partial{x_j}}x_i,\quad 1\le i<j\le n\right\}.
\end{equation}
Since $\bfz\in\crit_\RR(f)$, we have
\[
	\Gamma_\RR(f)=\left\{x\in\RR^n\ \Big|\ \frac{\partial f}{\partial x_i}x_j=
	\frac{\partial f}{\partial x_j}x_i,\ 1\le i<j\le n\right\}.
\]
The real variety $\Gamma_\RR(f)$ is called the {\itshape tangency variety} at the origin \cite{Durfee,VS}.
Geometrically, the tangency variety $\Gamma_\RR(f)$ consists of
all points $x$ in $\RR^n$ at which the level set of $f$ is tangent to the sphere in
$\RR^n$ centered at the origin with radius $\Vert x\Vert_2$.

\begin{prop}\label{prop::uv}
%Suppose that $\crit_\RR(f)\cap\mathsf{B}_r=\{\bfz\}$ for some $r\in\RR_+$.
%is the only critical point of $f$ in $\mathsf{B}_R$,
%Suppose that $R$ is an isolation radius of $f$.
%Then
For any $r\in\RR_+$ and $u\in\mathsf{B}_r$,
there exists a point $v\in\Gamma_\RR(f)$ with $\Vert
v\Vert_2\le\Vert u\Vert_2$ such that $f(v)=f(u)$.
\end{prop}
\begin{proof}
Consider the following optimization problem
\[
\underset{x\in\RR^n}{\min}\ \Vert x\Vert_2^2\quad\text{s.t.}\ f(x)=f(u).
\]
Since $u$ is a feasible point, there exists a minimizer $v$ with $\Vert
v\Vert_2\le\Vert u\Vert_2$.
If $\nabla f(v)=\bfz$, then clearly $v\in\Gamma_\RR(f)$. Otherwise,
%If $v\neq \bfz$, then
the linear independence constraint qualification condition holds at $v$
and therefore $v$ satisfies the Karush--Kuhn--Tucker optimality condition.
It implies that $v\in\Gamma_\RR(f)$.
\end{proof}
\begin{remark}
	For any $r\in\RR_+$, the semi-algebraic set $\Gamma_\RR(f)\cap\mathsf{B}_r$
	has finitely many connected components $K_i$ with $\bfz$ belonging to
	their closures by Proposition \ref{prop::connected}.
	For each $i$,
	by the curve selection lemma, there exists an analytic curve
	$\varphi_i: [0, \epsilon]\rightarrow\RR^n$ such that $\varphi_i(0)=\bfz$ and
	$\varphi_i(t)\in K_i$ for $t\neq 0$.
	%(assume that $\bfz$ belongs to the closure of $K_i$).
	By Proposition \ref{prop::uv}, it can
	be shown that the behavior of $f$ along the curves $\varphi_i$ captures
	all information of $f$ near $\bfz$. That is, we can identify the type of
	the critical point $\bfz$ of $f$ by extremal test of the univariate functions
	$f\circ\varphi_i$ at $0$.
	This approach was studied in \cite{Barone-Netto1984,Barone-Netto1996,VS} which,
	however, provide no general procedures to compute the expressions of the
	analytic functions $\varphi_i$.
\end{remark}

%For any $r\in\QQ$, let
%\begin{equation}\label{eq::op}
%f_r^{min}:=\underset{\Vert x\Vert_2\le r}{\min} f(x)\quad\text{and}\quad
%f_r^{max}:=\underset{\Vert x\Vert_2\le r}{\max} f(x)
%\end{equation}
For any $r\in\RR_+$,
recall the definition of $f_r^{\min}$ and $f_r^{\max}$ in $(\ref{eq::ops})$.
%We have
\begin{cor}\label{cor::uv}
%Suppose that $R$ is an isolation radius of $f$, then
For any $r\in\RR_+$, we have
\[
%\begin{aligned}
f_r^{\min}=\min\{f(x)\mid x\in \Gamma_\RR(f)\cap \mathsf{B}_r\}
\quad\text{and}\quad
f_r^{\max}=\max\{f(x)\mid x\in \Gamma_\RR(f)\cap \mathsf{B}_r\}.
%\end{aligned}
\]
\end{cor}
\begin{proof}
Since $f_r^{\min}$ and $f_r^{\max}$ can be reached,
the conclusion follows from Proposition \ref{prop::uv}.
\end{proof}

\begin{cor}\label{cor::notiso}
$\bfz$ is not isolated in $\Gamma_\RR(f)$ and $\dim(\Gamma_\RR(f)\backslash\crit_\RR(f))\ge 1$.
\end{cor}
\begin{proof}
%Without loss of generality, assume that
Since $f$ is not a zero polynomial,
either $f^{\max}_r$ or $f^{\min}_r$ is nonzero for any $r\in\RR_+$.
%otherwise, $f$ is a zero polynomial.
Hence by Proposition \ref{prop::uv}, there exists a nonzero
$u_r\in\Gamma_\RR(f)\cap\mathsf{B}_r$ such that either $f(u_r)=f^{\max}_r\neq 0$ or
$f(u_r)=f^{\min}_r\neq 0$.
%with $\Vert u\Vert_2\le\varepsilon$.
Thus, $\bfz$ is not isolated in $\Gamma_\RR(f)$ since $\lim_{r\rightarrow 0} u_r=\bfz$.

Note that
%there are only finitely many values in the set
$\{f(x)\mid x\in\crit_{\RR}(f)\}$ is a finite set by Sard's theorem.
Because $f(u_r)\neq 0$ for each $r$ and $\lim_{r\rightarrow 0} f(u_r)=0$,
there must be infinitely many $u_r\in\Gamma_\RR(f)\backslash\crit_\RR(f)$.
Then we have $\dim(\Gamma_\RR(f)\backslash\crit_\RR(f))\ge 1$ by Theorem
\ref{th::decomp}.
\end{proof}

%For any positive $r\in\RR$,
%denote $\mathsf{B}_r:=\{x\in\RR^n\mid \Vert x\Vert_2\le r\}$, i.e.,
%the ball centered at the origin with radius $r$.

\begin{define}\label{def::is}
We call a $R\in\RR_+$ an {\bf\itshape isolation radius} of $\bfz$
if $\crit_\RR(f)\cap\mathsf{B}_R=\{\bfz\}$.
\end{define}

\begin{define}\label{def::faithr}
We call an isolation radius $R$ of $\bfz$ a {\bf\itshape faithful radius}
if the following conditions hold:
%For a fixed $R\in\QQ$, if
%$\mathscr{C}(f)\cap
%\mathsf{B}_R=\{0\}$ and for any $u\in \Gamma_\RR(f)\cap\mathsf{B}_R$, there exists a
%differentiable function $\phi(t)=(\phi_1(t),\ldots,\phi_n(t))$ with $0\le t\le
%1$ such that
%\begin{enumerate}[\upshape 1.]
\begin{inparaenum}[\upshape(i\upshape)]
%\item $\crit_\RR(f)\cap \mathsf{B}_R=\{\bfz\}$;
\item $\Gamma_\RR(f)\cap\mathsf{B}_R$ is connected;
%for any $u\in \Gamma_\RR(f)\cap\mathsf{B}_R$, there exits a continuous
%path in $\Gamma_\RR(f)\cap\mathsf{B}_R$ connecting $0$ and $u$;
%\item for any $u\in \Gamma_\RR(f)\cap\mathsf{B}_R$, there exists a
%differentiable function $\phi(t)=(\phi_1(t),\ldots,\phi_n(t))$ and $0<\bar{t}<1$
%such that $\phi(t)\in \Gamma_\RR(f)\cap\mathsf{B}_R, 0<t<1$ and
%%\item $\phi(t)\subseteq \Gamma_\RR(f)$ for $0\le t\le 1$ with $\phi(0)=0$ and
%%$\phi(1)=u$.
%%\item For every $0<t<1$, we have
%\begin{equation}\label{eq::phi}
%\phi(\bar{t})=u,\qquad \frac{d\sum_{i=1}^n\phi^2_i}{dt}(\bar{t})\neq 0,
%\end{equation}
%\end{enumerate}
\item $\Gamma_\RR(f)\cap \bfV_{\RR}(f)\cap\mathsf{B}_R=\{\bfz\}$.
\end{inparaenum}
%then we call $R$ a {\bf\itshape faithful radius}.
\end{define}

%\subsection{Deciding type of critical point}

Note that $\Gamma_\RR(f)\cap\mathsf{B}_R$ is also path connected if $R$
is a faithful radius by Proposition \ref{prop::connected}.
The following result shows that if $R$ is a faithful radius, then we can
classify the degenerate real critical point $\bfz$ of $f$ by the signs of its
global extrema over the ball $\bfB_R$.

\begin{theorem}\label{th::main}
Suppose $R\in\RR_+$ is a faithful radius, then
\begin{enumerate}[\upshape (1)]
\item $\bfz$ is a local minimizer if and only if $f_R^{\max}>0$ and $f_R^{\min}=0$;
\item $\bfz$ is a local maximizer if and only if $f_R^{\max}=0$ and $f_R^{\min}<0$;
\item $\bfz$ is a saddle point if and only if $f_R^{\max}>0$ and $f_R^{\min}<0$.
\end{enumerate}
\end{theorem}
\begin{proof}
%By Corollary \ref{cor::uv},
$(1)$ and $(2)$ are clear if we can prove $(3)$.
%As $f_R^{min}<0$ can be be reached,
Since $\bfz$ is an isolated real critical point,
we only need to prove the ``if'' part.
By Corollary \ref{cor::uv}, there exists a
$u\in\Gamma_\RR(f)\cap \mathsf{B}_R$ such that $f(u)=f_R^{\min}$. Since $R$
is a faithful radius, $\bfz$ and $u$ are path connected, i.e. there exists a
continuous
%differentiable
mapping $\phi(t): [a,b]\rightarrow \Gamma_\RR(f)$ such that
$\bfz\not\in\phi((a,b))$,
%for $0\le t\le 1$ such that
$\phi(a)=\bfz$ and $\phi(b)=u$. We have $f(\phi(t))<0$ for all $t\in(a,b]$.
Otherwise, by the continuity, there exists $\bar{t}\in(a,b)$ such that
$f(\phi(\bar{t}))=0$. Since
$R$ is faithful, we have $\phi(\bar{t})=\bfz$ by the definition, a contradiction.
%which
%contradicts Proposition \ref{prop::0}.
Similarly, let $f_R^{\max}>0$ be reached at
$v\in \Gamma_\RR(f)\cap \mathsf{B}_R$, then there exists a continuous
%differentiable
mapping $\varphi(t): [a,b]\rightarrow\Gamma_\RR(f)$ such that
$\bfz\not\in\varphi((a,b))$,
%for $0\le t\le 1$
%such that
$\varphi(a)=\bfz$, $\varphi(b)=v$ and  $f(\varphi(t))>0$ for all $t\in(a,b]$.
Therefore, $\bfz$ is a saddle point of $f$.
\end{proof}

%For a polynomial $f\in\RR[X]$, if $\bfz$ is an isolated real critical point
%of $f$,
%$\crit_\RR(f)\cap\bfB_R=\{\bfz\}$ for some $R\in\RR_+$,
%then
There always exists a faithful radius of $\bfz$. In fact,
\begin{theorem}
%Let $f$ be a $C^1$-semi-algebraic function on $\mathbb{R}^n$.
%Then
$\bfz$ is an isolated real critical point of $f$ if and only if
there is a faithful radius of $\bfz$.
\end{theorem}
\begin{proof}
We only need to prove the ``only if'' part and
assume that $\bfz$ is an isolated real critical point of $f$.
%It is not hard to see that:

(i) By the assumption, there is an isolation radius $R_1\in\RR_+$ such that
$\crit_\RR(f)\cap \mathsf{B}_{R_1}=\{\bfz\}$.

(ii) Since $\Gamma_\RR(f)\cap\mathsf{B}_R$ is a closed semi-algebraic set,
by Proposition \ref{prop::connected},  it has finitely
many connected components $\mathcal{C}_1,\ldots,\mathcal{C}_s$
which are closed in $\RR^n$. Assume that the
components $\mathcal{C}_i$, $2\le i\le s$, do not contain $\{\bfz\}$. For each
$2\le i\le s$, since the component $\cal{C}_i$ is closed and bounded,
the function $\sum_{i=1}^n X_i^2$ reaches its minimum on
$\cal{C}_i$ at a minimizer $u^{(i)}\in\cal{C}_i$. Fix a $R_2\in\RR_+$ such
that $0<R_2<\min_{2\le i\le s}\Vert u^{(i)}\Vert_2$,
then $\Gamma_\RR(f)\cap\mathsf{B}_{R_2}$ is connected.

%(i) there exists $R_1 > 0$ such that $\mathrm{Crit}(f) \cap \mathsf{B}_{R_1} = \{0\}.$
%(ii) there exists $R_2 > 0$ such that $V_{\mathbb{R}} (\Gamma) \cap \mathsf{B}_{R_2}$ is connected.
%In fact, these two claims were shown in your draft!

(iii) We claim that there exists $R_3\in\RR_+$ such that
$\Gamma_\RR(f)\cap\bfV_\RR(f)\cap\mathsf{B}_{R_3}=\{\bfz\}$.
%(It is better to use the notation $V_{\mathbb{R}}
%({\color{red}\Gamma\cap\{f\}})$ instead of $V_{\mathbb{R}} (f, \Gamma)$).
Suppose to the contrary that such $R_3$ does not exist.
By the Curve Selection Lemma, we can find an analytic curve
$\phi \colon [0, \epsilon] \rightarrow \mathbb{R}^n$ such that
	$\phi(0)=\bfz, f(\phi(t)) = 0$ and $\phi(t) \in\Gamma_\RR(f)\backslash\{\bfz\}$
for all $t \in (0, \epsilon]$.
For each $t$, by the definition,
$\nabla f(\phi(t)) = \lambda (t) \phi(t)$ for some
$\lambda(t) \in \mathbb{R}$ and furthermore,
\begin{eqnarray*}
0=\frac{\ud(f\circ \phi)(t)}{\ud t} =
\left\langle \nabla f(\phi(t)), \frac{\ud\phi(t)}{\ud t}\right\rangle
=\lambda (t) \frac{\ud \|\phi(t)\|^2_2}{2\ud t}.
\end{eqnarray*}
By the monotonicity lemma,
$\lambda(t) = 0$ and hence $\nabla f(\phi(t)) = \bfz$ for $0 \le t \ll 1,$
which is a contradiction since $\bfz$ is an isolated real critical point of $f$.

Clearly, $R := \min\{R_1, R_2, R_3\}$ is a faithful radius of $f.$
\end{proof}

\section{Computational aspects of faithful radius}\label{sec::cmptr}
In this section, we present some computational criteria and an algorithm for computing
a faithful radius of the isolated real critical point $\bfz$ of the polynomial $f$.

\subsection{Curve of tangency}
We now recall some background about the tangency variety at a general point
which is studied in \cite{Durfee,VS}. For any $a\in\RR^n$, let
\[
\Gamma_\RR(f,a)=\left\{x\in\RR^n\mid \exists \lambda\in\RR\quad
	\text{s.t.}\ \nabla f(x)=\lambda (x-a)\right\}.
	%\frac{\partial{f}}{\partial{x_i}}(x_j-a_j)-
%\frac{\partial{f}}{\partial{x_j}}(x_i-a_i),\quad 1\le i<j\le n\right\}.
\]
In particular, $\Gamma_\RR(f)=\Gamma_\RR(f,\bfz)$.
Geometrically, the tangency variety $\Gamma_\RR(f,a)$ consists of
all points in $\RR^n$ at which the level set of $f$ is tangent to the sphere in
$\RR^n$ centered in $a$ with radius $\Vert x-a\Vert_2$.
%For any polynomial $g(X)\in\RR[X]$, a real or complex affine variety $V$ and the
%map $g: V\rightarrow \CC$,
%denote $\mathscr{C}(g, V)$ the set of critical points of $g$ on $V$. In
%perticular, let $\mathscr{C}(g)=\mathscr{C}(g,\RR^n)$ for simplicity.

%\subsection{Generic properties of tangency variety}
\begin{prop}\cite[Lemma 2.1]{VS}\label{prop::VS}
It holds that
\begin{enumerate}[\upshape (i)]
\item $\Gamma_\RR(f,a)$ is a nonempty, unbounded and semi-algebraic set;
\item There exists a proper algebraic set $\Omega\subseteq\RR^n$ such that for
each $a\in\RR\backslash\Omega$, the set
$\Gamma_\RR(f,a)\backslash\crit_\RR(f)$ is a one-dimensional submanifold of
$\RR^n$.
\end{enumerate}
\end{prop}
Therefore, $\Gamma_\RR(f,a)$ is also called {\itshape curve of tangency}.
Note that for the given $f\in\RR[X]$,
$\bfz$ might not belong to $\Omega$ as in Proposition \ref{prop::VS} and
then the statement $(ii)$ in Proposition \ref{prop::VS} is not necessarily true
for $\Gamma_\RR(f)=\Gamma_\RR(f,\bfz)$.
%(see \eqref{eq::gamma}).  %in general.
However, in the following we will show that
$\Gamma_\RR(f)\backslash\crit_\RR(f)$ is indeed a
one-dimensional semi-algebraic set of $\RR^n$ after a generic linear change of the
coordinates of $f$.

For $f\in\RR[X]$ and an invertible matrix
$A\in\RR^{n\times n}$, denote $f^A=f(Ax)$ the polynomial obtained
by applying the change of variables $A$ to $f$.
%For any $A\in GL_n(\RR)$, Let
%\[
%\Gamma^A=\left\{g^A_{i,j}(x)
%:=\frac{\partial{f^A}}{\partial{x_i}}x_j-\frac{\partial{f^A}}{\partial{x_j}}x_i,
%\Big | 1\le i<j\le n\right\}.
%\]
%Let $\mathcal{S}(V_{\RR}(\Gamma^A))$ be the set of singular points in $V_{\RR}(\Gamma^A)$.
Denote $\Sn\subset\RR^{n\times n}$ (resp. $\SSn\subset\RR^{n\times n}$) as the set of
symmetric (resp. positive definite) matrices with real number entries.
For any matrix $P\in\RR^{n\times n}$, define
\[
	\Gamma_\RR(f,P)=\{x\in\RR^n\mid \exists \lambda\in\RR,\quad
		\text{s.t.}\ \nabla f(x)=\lambda Px\}.
\]
%Clearly, for any $P\in\SSn$, there
%is a matrix $A\in GL_n(R)$ such that $P=A^TA$.
Given an invertible matrix $A\in\RR^{n\times n}$ and a subset $S\subseteq\RR^n$,
let
\[
	A(S)=\{Ax\mid x\in S\}.
\]
\begin{lemma}\label{lem::trans}
	Given an invertible matrix
	$A\in\RR^{n\times n}$, let $P=A^{-T}A^{-1}$, then we have
	$\Gamma_\RR(f^A)=A^{-1}(\Gamma_\RR(f,P))$ and $\crit_\RR(f^A)=A^{-1}(\crit_\RR(f))$.
\end{lemma}
\begin{proof}
By the definition, we have
\[
\begin{aligned}
	\Gamma_\RR(f^A)&=\{x\in\RR^n\mid \exists \lambda\in\RR\quad\text{s.t.}\
\nabla f^A(x)=\lambda x\}\\
	&=\{x\in\RR^n\mid \exists \lambda\in\RR\quad\text{s.t.}\
A^T\nabla f(Ax)=\lambda x\}\\
&=\{A^{-1}y\in\RR^n\mid \exists \lambda\in\RR\quad\text{s.t.}\
\nabla f(y)=\lambda A^{-T}A^{-1}y\}\\
&=A^{-1}(\Gamma_\RR(f,P)).\\
%	&=A^{-1}\cdot\{x\in\RR^n\mid \exists \lambda\in\RR\quad\text{s.t.}\
%\nabla f(x)=\lambda Px\}=A^{-1}\cdot\Gamma_\RR(f,P).\\
\end{aligned}
\]
Similarly, it holds that $\crit_\RR(f^A)=A^{-1}(\crit_\RR(f))$.
\end{proof}

Let $\mathcal{I}^{n \times n}$ be the set of all invertible $n\times n$ matrices in
$\RR^{n\times n}$.

\begin{theorem}\label{th::dim}
There exists an open and dense semi-algebraic set $\mathcal{U} \subset \mathcal{I}^{n \times n}$
such that for all $A \in \mathcal{U},$
the set $\Gamma_\RR(f^A)\backslash\crit_\RR(f^A)$ is a one-dimensional
semi-algebraic set of $\RR^n$.
\end{theorem}
\begin{proof}
Clearly, $\SSn$ is an open semi-algebraic subset of $\mathcal{S}^{n \times n} \equiv \mathbb{R}^{\frac{n(n+1)}{2}},$
where we identify $P := (p_{ij})_{n \times n}\in\mathcal{S}^{n \times n}$
with
\[
(p_{11}, \ldots, p_{1n}, p_{22}, \ldots, p_{2n},  \ldots, p_{nn}) \in \mathbb{R}^{\frac{n(n+1)}{2}}.
\]

% By the Cholesky decomposition, it is easy to see that  $\sfu_n(\RR)$ is semi-algebraically homeomorphic to $\SSn.$
% By Lemma~\ref{lem::trans},

We first show that $\Gamma_\RR(f,P)\backslash\crit_\RR(f)$ is
%one-dimensional
a semi-algebraic set of dimension $\le 1$ for almost every $P\in\SSn$.
To do this, we consider the semi-algebraic map
\[
	\begin{aligned}
		F:\ (\RR^n\backslash\crit_\RR(f))\times\RR \times \SSn &\ \rightarrow\ &&\RR^n\\
		(x, \lambda, P)&\ \mapsto \ &&\nabla f(x)-\lambda Px.
	\end{aligned}
\]
We will show that
$\bfz\in\RR^n$ is a regular value of the map $F$.
Take any $(x,\lambda,P)\in F^{-1}(\bfz)$, then $x\neq\bfz$ and
$\lambda\neq 0$. Otherwise, we have $\nabla f(x)=\bfz$ and hence $x\in\crit_\RR(f)$,
a contradiction.
Without loss of generality, we assume that $x_1\neq 0$.
%for some $i_0\ge 1$.
%Denote $p_{i,j}$ as the entry of $P$ at the $i$-th row
%and $j$-th column.
Note that $p_{i j}=p_{j i}$.
Then, a direct computation shows that the Jacobian matrix
$\jac(F)$ of the map $F$ contains the following columns
\[
	-\lambda\cdot\left[
\begin{array}{ccccc}
	x_1 &  x_2 &  x_3 & \cdots &  x_n\\
	0	    &  x_1 & 0           & \cdots & 0\\
	0           & 0           &  x_1 & \cdots & 0\\
	\vdots      & \vdots      & \vdots      &        & \vdots\\
	0           & 0           & 0           & \cdots & x_1\\
\end{array}
\right],
\]
which correspond to the partial derivatives of $F$ with respect to the
variables $p_{1 j}$ for $j=1,\ldots,n$.
Therefore, for all $(x, \lambda, P)\in F^{-1}(\bfz)$, we have
$\rank(\jac(F))=n$ and hence $\bfz$ is a regular value of $F$.
By Thom's weak transversality theorem (\cite{DT}, \cite[Theorem~1.10]{HPh}),
there exists a semi-algebraic set $\Sigma \subset \mathcal{S}^{n \times n}_{++}$ of
dimension $< \frac{n(n+1)}{2}$ such that
for all $P \in \mathcal{S}^{n \times n}_{++} \setminus \Sigma,$  $\bfz$ is a regular value
of the map
\[
\begin{aligned}
	F_P:\ (\RR^n\backslash\crit_\RR(f))\times\RR&\ \rightarrow\ &&\RR^n\\
	(x, \lambda)&\ \mapsto \ &&F(x,\lambda,P).
\end{aligned}
\]
Thus, $F_P^{-1}(\bfz)$ is either empty or a one-dimensional submanifold of $\RR^n$.
Since $\Gamma_\RR(f,P)\backslash\crit_\RR(f)$ is the projection of $F_P^{-1}(\bfz)$
on the first $n$ coordinates, by \cite[Proposition 2.8.6]{realAG},
we have $\dim(\Gamma_\RR(f,P)\backslash\crit_\RR(f))\le 1$.
%It is clear that
%$\dim(\Gamma_\RR(f,P)\backslash\crit_\RR(f))\ge 1$ by
%Proposition \ref{prop::uv},
%hence we obtain that $\dim(\Gamma_\RR(f,P)\backslash\crit_\RR(f))=1$.

Next, it is easy to see that
$$\mathcal{S}^{n \times n}_{++} \rightarrow \mathcal{S}^{n \times n}_{++}, \quad P \mapsto P^{-1},$$
 is a semi-algebraic homeomorphism. Hence,
 $$\Sigma^{-1} := \{P^{-1} \in \mathcal{S}^{n \times n}_{++}  \ | \ P \in \Sigma\} \subset \mathbb{R}^{\frac{n(n+1)}{2}}$$
 is a semi-algebraic set of dimension $< \frac{n(n+1)}{2}.$
 Consequently, by \cite[Lemma 1.4]{HPh},
 there exists a non-constant polynomial
 $\mathcal{F} \colon \mathbb{R}^{\frac{n(n+1)}{2}} \rightarrow \mathbb{R}$ such that
$$\Sigma^{-1}  \subset \{Q \in \mathcal{S}^{n \times n}_{++}  \ | \  \mathcal{F}(Q)  = 0\}.$$
Note that the corresponding
$$\mathcal{I}^{{n \times n}} \rightarrow \mathcal{S}^{n \times n}_{++}, \quad A \mapsto A A^T,$$
is a polynomial map. Thus,
$$\{A \in \mathcal{I}^{{n \times n}} \ | \ \mathcal{F}(AA^T) = 0\}$$
is an algebraic set. It follows that
$\mathcal{U} := \{A \in \mathcal{I}^{{n \times n}} \ | \ \mathcal{F}(AA^T) \ne 0\}$
is an open and dense semi-algebraic subset of $\mathcal{I}^{n \times n}.$
Furthermore, by the definition, for all $A \in \mathcal{U}$,
we have $P := (AA^T)^{-1} \in  \mathcal{S}^{n \times n}_{++} \setminus \Sigma$ and hence,
$\dim(\Gamma_\RR(f^A)\backslash\crit_\RR(f^A))\le 1$ by Lemma~\ref{lem::trans}.
Since $A\in\mathcal{U}$ is invertible, we have $\dim(\Gamma_\RR(f^A)\backslash\crit_\RR(f^A))\ge 1$
by Corollary \ref{cor::notiso} and then the conclusion follows.
%the set $\Gamma_\RR(f^A)\backslash\crit_\RR(f^A)$ is a one-dimensional  semi-algebraic set of $\RR^n.$
\end{proof}

\begin{remark}{\rm
(i) The technique of a generic linear change of variables was also used in 
\cite{SafeyElDin2003} to show the dimension of polar varieties;
(ii) We may also use the new inner product
$\langle x, x' \rangle_P := \langle Px, Px' \rangle$
(and the corresponding  norm $\sqrt{\langle x, x' \rangle_P}$)
for some generic $P\in\mathcal{S}^{n \times n}_{++}$
instead of using a generic linear change of variables $y= Ax.$
In fact, it is not hard to see that with this new inner product,
$\Gamma_\RR(f)\backslash\crit_\RR(f)$ is also a one-dimensional
semi-algebraic set of $\RR^n.$
}\end{remark}

We illustrate the result in Theorem \ref{th::dim} by the following simple example.
\begin{example}
Consider the polynomial $f(X_1,X_2)=X_1^2+X_2^2$.
We have $\Gamma_\RR(f)=\RR^2$. However, if we make a linear
change of variables and let $f^A=(a_{1,1}X_1+a_{1,2}X_2)^2+(a_{2,1}X_1+a_{2,2}X_2)^2$,
then
$\Gamma_\RR(f^A)=\{(x_1,x_2)\in\RR^2\mid
	(a_{1,1}a_{1,2}+a_{2,1}a_{2,2})(x_2^2-x_1^2)
	+(a_{1,1}^2+a_{2,1}^2-a_{1,2}^2-a_{2,2}^2)x_1x_2=0\}$.
Clearly, $\dim(\Gamma_\RR(f^A)\backslash\crit_\RR(f^A))=1$ whenever
$a_{1,1}a_{1,2}+a_{2,1}a_{2,2}\neq 0$ or
$a_{1,1}^2+a_{2,1}^2-a_{1,2}^2-a_{2,2}^2\neq 0$.
\end{example}

For $f\in\RR[X]$ and any matrix $P\in\RR^{n\times n}$, let
\begin{equation}
\begin{aligned}
\Gamma_\CC(f)&:=\left\{x\in\CC^n\mid \exists \lambda\in\CC
	\quad\text{s.t.}\ \nabla f(x)=\lambda x\right\},\\
\Gamma_\CC(f,P)&:=\left\{x\in\CC^n\mid \exists \lambda\in\CC
	\quad\text{s.t.}\ \nabla f(x)=\lambda Px\right\}.
\end{aligned}
%	\frac{\partial{f}}{\partial{x_i}}x_j-
%\frac{\partial{f}}{\partial{x_j}}x_i,\quad 1\le i<j\le n\right\}.
\end{equation}
Recall that $\crit_\CC(f)$ and $\crit_\CC(f^A)$ denote the sets of complex
critical points of $f$ and $f^A$, respectively.
As in Lemma \ref{lem::trans}, we still have
\begin{lemma}\label{lem::transC}
	Given an invertible matrix
	$A\in\RR^{n\times n}$, let $P=A^{-T}A^{-1}$, then we have
	$\Gamma_\CC(f^A)=A^{-1}(\Gamma_\CC(f,P))$ and $\crit_\CC(f^A)=A^{-1}(\crit_\CC(f))$.
\end{lemma}

\begin{cor}\label{cor::dim1}
There exists an open and dense semi-algebraic set
$\mathcal{U} \subset \mathcal{I}^{n \times n}$ such that for all
$A \in \mathcal{U},$
%For almost every matrix $A\in\sfu_n(\RR)$,
the Zariski closure
${\overline{\Gamma_\CC(f^A)\backslash\crit_\CC(f^A)}}^\mathcal{Z}$
is a one-dimensional algebraic variety in $\CC^n$.
\end{cor}
\begin{proof}
Denote $\mathcal{S}_\CC^{n\times n}$ as the set of symmetric matrices
in $\CC^{n\times n}$, which can be identified with the space $\CC^{\frac{n(n+1)}{2}}$.
By similar arguments as in Theorem \ref{th::dim}, it is easy to see that
$\bfz$ is a regular value of the map
\[
	\begin{aligned}
		F:\ (\CC^n\backslash\crit_\CC(f))\times\CC\times
		\mathcal{S}_\CC^{n\times n}&\ \rightarrow\ &&\CC^n\\
		(x, \lambda, P)&\ \mapsto \ &&\nabla f(x)-\lambda Px.
	\end{aligned}
\]
Then according to Thom's weak transversality theorem
(cf. \cite{DT}, \cite[Theorem~1.10]{HPh}, and \cite[Proposition B.3]{SS2013}),
there exists a Zariski closed subset
$\Sigma_\CC\subset\mathcal{S}_\CC^{n\times n}$ such that for all
$P\in\mathcal{S}_\CC^{n\times n}\backslash\Sigma_\CC$,
$\bfz$ is a regular value of the map
\[
\begin{aligned}
	F_P:\ (\CC^n\backslash\crit_\CC(f))\times\CC&\ \rightarrow\ &&\CC^n\\
	(x, \lambda)&\ \rightarrow\ &&F(x,\lambda,P).
\end{aligned}
\]
%for any fixed $P\in\mathcal{S}_\CC^{n\times n}$.
%Similar to the proof of Theorem \ref{th::dim}, using arguments on complex
%manifolds, we have that
It follows that $F^{-1}_P(\bfz)$ is either empty or a one-dimensional
%complex submanifold
quasi-affine set of $\CC^{n+1}$.
%for almost every $P\in\mathcal{S}_\CC^{n\times n}$.
%Note that
Since $\overline{\Gamma_\CC(f,P)\backslash\crit_\CC(f)}^{\mathcal{Z}}$
is the Zariski closure of
%the projection of
%Note that $\Gamma_\CC(f,P)\backslash\crit_\CC(f)$
the projection of $F^{-1}_P(\bfz)$ on the first $n$ coordinates,
%Since
%and $\dim(\overline{
%\Gamma_\CC(f,P)\backslash\crit_\CC(f)}^{\mathcal{Z}})\ge 1$ by
%Proposition \ref{prop::uv},
we have $\dim(\overline{
\Gamma_\CC(f,P)\backslash\crit_\CC(f)}^{\mathcal{Z}})\le 1$ for
all $P\in\mathcal{S}_\CC^{n\times n}\backslash\Sigma_\CC$.
Let $\Sigma=\Sigma_\CC\cap\SSn$, then it is clear that
$\Sigma\subset\SSn$ is a semi-algebraic set of dimension $<\frac{n(n+1)}{2}$
and $\dim(\overline{
\Gamma_\CC(f,P)\backslash\crit_\CC(f)}^{\mathcal{Z}})\le 1$ for
all $P\in\SSn\backslash\Sigma$.
%$P\in\mathcal{S}_\CC^{n\times n}$.
%Since the set $\mathcal{S}_{++}^{n\times n}\subset\RR^{n\times n}$,
%considered as a subset of $\CC^{n\times n}$, is submanifold of
%$\mathcal{S}_\CC^{n\times n}$ of the same dimension,
%it holds that $\dim(\overline{
%\Gamma_\CC(f,P)\backslash\crit_\CC(f)}^{\mathcal{Z}})=1$ for almost every
%$P\in\SSn$.
Hence, the conclusion follows by similar arguments as in the proof
of Theorem \ref{th::dim}.
%by Lemma \ref{lem::transC} and the
%Cholesky decompositions of matrices in $\SSn$.
\end{proof}

%\begin{cor}\label{cor::iredim}
%Suppose $0$ is an isolated critical point of $f(x)\in\RR[X]$. After a generic
%invertible change of the coefficients of $f(x)$, there exists a $R\in\RR$ such
%that the semialgebraic set $\Gamma_\RR(f)\cap\mathsf{B}_R$ is of dimension
%$1$.
%%for every irreducible component
%%$V$ of $\Gamma_\RR(f)$ with $0\in V$, we have $\dim(V)\le 1$.
%\end{cor}
%\begin{proof}
%%Assume that the coefficients of $f$ are generic. Consider the irreducible
%%decomposition $V_\CC(\Gamma)=Z_1\cup\cdots\cup Z_s\cup\cdots\cup Z_m$ where
%%$Z_i\not\subset\crit(f)$ for $1\le i\le s$ and $Z_i\subset\crit(f)$ for $s+1\le
%%i\le m$. Then, $\overline{V_{\CC}(\Gamma)\backslash\crit(f)}=Z_1\cup\cdots\cup
%%Z_s$. Since $\bfz$ is isolated critical point, we have
%%$\bfz\not\in Z_{s+1}\cup\cdots\cup Z_m$. Thus, there exists a $R\in\RR$ such
%%that $\Gamma_\RR(f)\cap\mathsf{B}_R=(Z_1\cup\cdots\cup Z_s)\cap\mathsf{B}_R$.
%%%Since $\dim(Z_1\cup\cdots\cup Z_s)=1$ by the theorem,
%%The conclusion follows by the theorem.
%%By Theorem \ref{th::dim}, the semialgebraic set
%%$\Gamma_\RR(f)\backslash\mathscr{C}(f)$
%%has dimension one after a generic invertible change of the coefficients
%%of $f(x)$. Since $0$ is an isolated critical point, there exists a $R_1\in\QQ$
%%such that $\mathscr{C}(f)\cap B(0,R_1)=\{0\}$ which implies the semialgebraic
%%set
%%$(\Gamma_\RR(f)\cap B(0,R_1)$ has dimension $\le 1$. Therefore, for every
%%irreducible component
%%$V$ of $\Gamma_\RR(f)$ with $0\in V$, we have $\dim(V)\le 1$.
%\end{proof}

\subsection{Sufficient criteria for faithful radius}

For a given $\mathscr{R}\in\RR_+$, we consider the following condition
\begin{condition}\label{con::curve}
For any $\bfz\neq u\in \Gamma_\RR(f)$ with $\Vert u\Vert_2<\mathscr{R}$,
there exist a neighborhood $\mathcal{O}_u\subset\mathsf{B}_\mathscr{R}$
of $u$,
a differentiable map $\phi(t): (a,b)\rightarrow\RR^n$ and $\bar{t}\in(a,b)$
such that $\phi((a,b))=\Gamma_\RR(f)\cap\mathcal{O}_u$,
$\phi(\bar{t})=u$ and
\begin{equation}\label{eq::phi}
%\phi(\bar{t})=u,\qquad
	\frac{\ud \left( \sum_{i=1}^n\phi^2_i \right)}{\ud t}(\bar{t})\neq 0.
\end{equation}
%
%\begin{prop}\label{prop::0}
%Suppose $R\in\RR$ is a faithful radius, then for every
%$0\neq u\in V_\RR(\Gamma)\cap \mathsf{B}_R$, we have $f(u)\neq 0$.
\end{condition}

\begin{theorem}\label{th::main2}
	Suppose that $\mathscr{R}\in\RR_+$ satisfies Condition \ref{con::curve}
	and $0<R<\mathscr{R}$, then $\Gamma_\RR(f)\cap\mathsf{B}_R$ is connected.
Moreover, if $R$ is an isolation radius,
%$\crit_\RR(f)\cap\mathsf{B}_R=\{\bfz\}$,
then
$\Gamma_\RR(f)\cap \bfV_{\RR}(f)\cap\mathsf{B}_R=\{\bfz\}$ and hence $R$ is a faithful
radius.
%for
%any $\bfz\neq u\in \Gamma_\RR(f)\cap\mathsf{B}_R\backslash\mathsf{S}_R$,
%there exist a neighborhood $\mathcal{O}$ of $u$ and
%%a continuous path in $\Gamma_\RR(f)\cap\mathsf{B}_R$ connecting $0$ and $u$;
%%for any $u\in \Gamma_\RR(f)\cap\mathsf{B}_R$, there exists a
%a differentiable function $\phi(t): (a,b)\rightarrow\RR^n$ and $\bar{t}\in(a,b)$
%such that $\phi((a,b))=\Gamma_\RR(f)\cap\mathsf{B}_R\cap\mathcal{O}$,
%$\phi(\bar{t})=u$ and
%%and $\phi(t)\subseteq \Gamma_\RR(f)$ for $0\le t\le 1$ with $\phi(0)=0$ and
%%$\phi(1)=u$.
%%For every $0<t<1$, we have
%\begin{equation}\label{eq::phi}
%%\phi(\bar{t})=u,\qquad
%\frac{d\sum_{i=1}^n\phi^2_i}{dt}(\bar{t})\neq 0.
%\end{equation}
%%
%%\begin{prop}\label{prop::0}
%%Suppose $R\in\RR$ is a faithful radius, then for every
%%$0\neq u\in V_\RR(\Gamma)\cap \mathsf{B}_R$, we have $f(u)\neq 0$.
%then $R$ is a faithful radius.
%$V_{\RR}(f, \Gamma)\cap\mathsf{B}_R=\{\bfz\}$.
\end{theorem}
\begin{proof}
%We need to prove that $\Gamma_\RR(f)\cap\mathsf{B}_R$ is connected and
%$\Gamma_\RR(f)\cap V_{\RR}(f)\cap\mathsf{B}_R=\{\bfz\}$.
	Suppose that $\Gamma_\RR(f)\cap\mathsf{B}_R$ is not connected, then it
	has a connected component $\mathcal{C}$ such that $\bfz\not\in\mathcal{C}$.
%By \cite[Theorem 2.4.5]{realAG}, $\Gamma_\RR(f)\cap\mathsf{B}_R$ has finitely
%many connected components $\mathcal{C}_1,\ldots,\mathcal{C}_s$. Assume that the
%components $\mathcal{C}_i$, $2\le i\le s$, do not contain $\{\bfz\}$. For each
%$2\le i\le s$,
	Since $\Gamma_\RR(f)\cap\mathsf{B}_R$ is closed, $\cal{C}$ is closed by
	Proposition \ref{prop::connected}. Then,
	%\cite[Theorem 2.4.4]{realAG},
	the function $\Vert X\Vert_2^2$ reaches its minimum on
$\cal{C}$ at a minimizer $u\in\cal{C}$. By the assumption, there exist
a neighborhood $\mathcal{O}_u$ of $u$ and
a differentiable mapping $\phi(t): (a,b)\rightarrow\RR^n$ and
$\bar{t}\in(a,b)$ such that
$\phi((a,b))=\Gamma_\RR(f)\cap\mathcal{O}_u$ and
$\phi(\bar{t})=u$.
By choosing $a, b$ near enough to $\bar{t}$, we may assume that
%Since $\phi((a,b))$ is connected,
$\phi((a,b))\subseteq\mathcal{C}\cap\mathcal{O}_u$.
Then, the function $\sum_{i=1}^n\phi_i^2$ reaches its local
minimum at $\bar{t}$, which contradicts $(\ref{eq::phi})$. Hence,
$\Gamma_\RR(f)\cap\mathsf{B}_R$ is connected.

Now suppose that $R$ is also an isolation radius. Assume to the contrary that
there exists
$\bfz\neq v\in \Gamma_\RR(f)\cap \bfV_\RR(f)\cap \mathsf{B}_R.$
%By Theorem \ref{th::R},
Since $\Gamma_\RR(f)\cap\mathsf{B}_R$ is connected, there exists a path
connecting $\bfz$ and $v$. Then, $f$ has a local extremum on a relative
interior of this path, say $u$. By the assumption,
there exists a differentiable mapping $\phi(t)$ on $(a,b)$ and
$\bar{t}\in(a,b)$ as described in the statement.
%with
%$\phi(t)\subseteq V_\RR(\Gamma)$ for $0\le t\le 1$ such that $\phi(0)=0$,
%$\phi(1)=u$.
Then the differentiable function $f(\phi(t))$ reaches a local extremum at
$\bar{t}$.  By the mean value theorem,
$$0=\frac{\ud f(\phi)}{\ud t}(\bar{t}).$$
On the other hand, since $R$ is an isolation radius,
$\phi(\bar{t}) \in \Gamma_\RR(f) \setminus\crit_\RR(f)$ and
hence there exists $\lambda \ne 0$ such that
$$\frac{\partial f}{\partial x_i}(\phi(\bar{t})) = \lambda \phi_i(\bar{t}), \quad \textrm{ for } i = 1, \ldots, n.$$
Therefore,
\[
0=\frac{\ud f(\phi)}{\ud t}(\bar{t})=\sum_{i=1}^n \frac{\partial f}{\partial x_i}(\phi(\bar{t})) \frac{\ud\phi_i}{\ud t}(\bar{t})
=\sum_{i=1}^n \lambda \phi_i(\bar{t}) \frac{\ud\phi_i}{\ud t}(\bar{t})
=\lambda \frac{\ud \left( \sum_{i=1}^n\phi^2_i \right)}{2\ud t}(\bar{t}),
\]
%\begin{eqnarray*}
%0 \ = \  \frac{\ud f(\phi)}{\ud t}(\bar{t}) &=& \sum_{i=1}^n \frac{\partial f}{\partial x_i}(\phi(\bar{t})) \frac{\ud\phi_i}{\ud t}(\bar{t}) \\
%&=& \sum_{i=1}^n \lambda \phi_i(\bar{t}) \frac{\ud\phi_i}{\ud t}(\bar{t}) \\
%&=& \lambda \frac{\ud \left( \sum_{i=1}^n\phi^2_i \right)}{2\ud t}(\bar{t}),
%\end{eqnarray*}
which contradicts (\ref{eq::phi}). Therefore $\Gamma_\RR(f)\cap\bfV_\RR(f)\cap\mathsf{B}_R=\{\bfz\}.$
%
%
%By the mean value theorem,
%%there exists $\bar{t}\in(a,b)$ such that
%\[
%0=\frac{\ud f(\phi)}{\ud t}(\bar{t})
%=\sum_{i=1}^n\frac{\partial f}{\partial x_i}(\phi(\bar{t}))
%\frac{\ud\phi_i}{\ud t}(\bar{t})
%\]
%Since $R$ is an isolation radius,
%$\phi(\bar{t})\not\in\crit_\RR(f)$ and there exist $i_0$ such that
%$\frac{\partial f}{\partial
%x_{i_0}}(\phi(\bar{t}))\neq 0$.
%Because $\phi(\bar{t})\in \Gamma_\RR(f)$, by the
%definition of $\Gamma_\RR(f)$,
%\[
%0=\sum_{i=1}^n\frac{\partial f}{\partial x_i}(\phi(\bar{t}))
%\frac{\ud\phi_i}{\ud t}(\bar{t})=\sum_{i=1}^n\frac{\partial f}{\partial x_i}(\phi(\bar{t}))\phi_{i_0}(\bar{t})
%\frac{\ud\phi_i}{\ud t}(\bar{t})=\sum_{i=1}^n\frac{\partial f}{\partial x_{i_0}}(\phi(\bar{t}))\phi_{i}(\bar{t})
%\frac{\ud\phi_i}{\ud t}(\bar{t})
%\]
%which implies
%\begin{equation}\label{eq::sum}
%\sum_{i=1}^n\phi_{i}(\bar{t})\frac{\ud\phi_i}{\ud
%t}(\bar{t})=\frac{\ud\sum_{i=1}^n\phi^2_i}{2\ud t}(\bar{t})=0.
%\end{equation}
%It contradicts (\ref{eq::phi}). Therefore, we get
%$\Gamma_\RR(f)\cap\bfV_\RR(f)\mathsf{B}_R=\{\bfz\}$.
%%the conclusion in Theorem \ref{th::R}.
\end{proof}

According to Theorem \ref{th::main2},
if we can compute a $\mathscr{R}\in\RR_+$ satisfying Condition
\ref{con::curve} and an isolation radius $R$ of $\bfz$ is given, then any $r\in\RR_+$
with $r<\min\{R,\mathscr{R}\}$ is a faithful radius of $\bfz$.
Hence, we next show that how to compute such a $\mathscr{R}$.

For a given ideal
%$\langle\bfg\rangle=\langle g_1,\ldots,g_s\rangle
$I\subseteq\RR[X]$ with $\dim(I)=1$,
compute its equidimensional decomposition
%$\langle\bfg\rangle
$I=I^{(0)}\cap I^{(1)}$ where $\dim(I^{(i)})=i$
for $i=0,1$. Compute the radical ideal
$\sqrt{I^{(1)}}=\langle g_1,\ldots,g_s\rangle$ with generators
$g_1,\ldots,g_s\in\RR[X].$ Note that there are efficient algorithms for the equidimensional
decomposition of $I$ such that
%if $I$ is radical, then
$I^{(0)}$ and $I^{(1)}$ are themselves radical (c.f. \cite[Section 3]{AUBRY2002543}
and \cite[Algorithm 4.4.9]{singular}).
%computed by \cite[Algorithm 4.4.9]{singular} are also radical.
%Let $\jac\left(g_1,\ldots,g_s,\Vert X\Vert_2^2\right)$ deonte the $(s+1)\times n$
%Jacobian matrix of $\{g_1,\ldots,g_s,\Vert X\Vert_2^2\}$.
Recall that $\Vert X\Vert_2^2=X_1^2+\cdots+X_n^2$.
Denote $\mathscr{D}$ as the set of the determinants of
the Jacobian matrices $\jac\left(g_{i_1},\ldots,g_{i_{n-1}},\Vert X\Vert_2^2\right)$
for all $\{i_1,\ldots,i_{n-1}\}\subset\{1, \ldots, s \}$.
(Note that $s \ge n - 1$ because $\dim(I^{(1)}) = 1.$) Define
\begin{equation}\label{eq::g}
	\begin{aligned}
		\mathscr{R}_{I^{(0)}}
		&:=\min\{r\in\RR\backslash\{0\}\mid
			%\bfV_\CC(I^{(0)}+\langle\Vert X\Vert_2^2-r^2\rangle)\neq\emptyset, r\in\RR\backslash\{0\}
			\exists x\in\bfV_\CC(I^{(0)}), \text{s.t. } x_1^2+\cdots+x_n^2=r^2\},\\
		\Delta_{I^{(1)}}&:=\{g_1,\ldots,g_s\}\cup\mathscr{D},\\
	%.\\
	%	&\left. \text{ corresponding to its
	%	$n\times n$ submatrices containing the last row}\right\},\\
%I_{\Delta_\bfg,Y}&:=\langle\Delta_\bfg, \Vert X\Vert_2^2-Y\rangle\cap\RR[Y],\\
		\mathscr{R}_{I^{(1)}}&:=\inf\{r\in\RR\backslash\{0\}\mid
		%\bfV_\CC(\langle\Delta_{I^{(1)}}\rangle+\langle\Vert X\Vert_2^2-r^2\rangle)\neq\emptyset,
		\exists x\in\bfV_\CC(\Delta_{I^{(1)}}), \text{s.t. } x_1^2+\cdots+x_n^2=r^2\},\\
	\mathscr{R}_I&:=\min\{\mathscr{R}_{I^{(0)}}, \mathscr{R}_{I^{(1)}}\}.
	%\left\{
	%\begin{aligned}
	%	&\min\{\Vert x\Vert_2\mid x\in\bfV_\RR(I^{(1)}), x\neq\bfz\}
	%	\quad& &\text{if}\ I_{\Delta_\bfg,Y}\neq\emptyset,\\
	%	&0\quad& &\text{otherwise.}
	%\end{aligned}\right.\\
\end{aligned}
\end{equation}
%\begin{remark}\label{rmk::less}
%	Note that when computing $\bfV_\RR(\Delta_{I^{(1)}})$, we can remove all
%	$n$-minors of $\jac\left(g_1,\ldots,g_s\right)$ from minors in $\Delta_{I^{(1)}}$.
%	In fact, since $I^{(1)}$ is equidimensional one, we have
%	$\det(\jac(g_1,\ldots,g_s))<n$ for all $x\in\bfV_\RR(\langle g_1,\ldots,g_s\rangle)$.
%\end{remark}

We have (see also \cite[Corollary~2.8]{Milnor1968}).
\begin{lemma}\label{lem::cs}
If $\dim(I)=1$,
%Let $I_Y=\langle\Delta_\bfg, \Vert X\Vert_2^2-Y\rangle\cap\RR[Y]$,
then $\mathscr{R}_I>0$.
%\begin{enumerate}[\upshape (i)]
%	\item $I_{\Delta_\bfg,Y}\neq\emptyset$ and hence $\mathscr{R}_\bfg>0$.
%	%Let $0<r_1<\cdots<r_s$ be the real zeros of $I_{\Delta_\bfg,Y}$,
%%then $r_1\le\mathscr{R}_\bfg$;
%	\item For any $0<R<\mathscr{R}_\bfg$,
%	the system $\bfV_\CC(\langle\bfg\rangle)\cap\mathsf{S}_{R,\CC}$
%	is zero-dimensional.
%\end{enumerate}
\end{lemma}
\begin{proof}
%(i).
	Since $I^{(0)}$ is zero-dimensional, we have $\mathscr{R}_{I^{(0)}}>0$.
	Now we show that $\mathscr{R}_{I^{(1)}}>0$.
Consider the map 	
\[
\begin{aligned}
	\Phi:\ \bfV_\CC(I^{(1)})&\ \rightarrow\ &&\CC\\
	x&\ \mapsto \ &&x_1^2+\cdots+x_n^2.
\end{aligned}
\]
Since $\langle g_1,\ldots,g_s\rangle$ is radical and equidimensional one,
$\bfV_\CC(\Delta_{I^{(1)}})$ consists of the singular locus of
$\bfV_\CC(I^{(1)})$ and the set of critical points of $\Phi$.
Since $\dim(\bfV_\CC(I^{(1)}))=1$, its singular locus is zero-dimensional.
By the algebraic Sard's theorem \cite{BAG},
there are only finitely many critical values
of the map $\Phi$ (note that if $\Phi$ is not dominant, the conclusion is clearly true).
Hence, the set $\Phi(\bfV_\CC(\Delta_{I^{(1)}}))$ is finite and $\mathscr{R}_{I^{(1)}}>0$.
%and $I_{\Delta_\bfg,Y}\neq\emptyset$
%by the Closure Theorem \cite[Chap. 3, \S 2, Theorem 3]{CLO}.
%Clearly, we have $r_1\le\mathscr{R}_\bfg$.
%(ii). Let $V_1\cup\cdots\cup V_s$ be the decomposition of
%$\bfV_\CC(\langle\bfg\rangle)$ as a union of irreducible components.
%Fix an $1\le i\le s$.
%If $V_i\cap\mathsf{S}_{R,\CC}\neq\emptyset$,
%then we show that $V_i\not\subseteq\mathsf{S}_{R,\CC}$.
%Indeed, by the part (i), we have $V_i\cap V_j=\emptyset$ for each $j\neq i$
%since $V_i\cap V_j$ is contained in the singular locus of
%$\bfV_\CC(\langle\bfg\rangle)$.
%Then for any $p\in V_i$, it holds that $T_p(V_i)=T_p(V)$ where $T_p(V_i)$
%and $T_p(V)$ denote the tangent spaces of $V_i$ and $V$ at $p$, respectively.
%Then,
%for any $h\in{\bf I}(V_i)$, the differiential of $h$ at $p$ can be expressed as
%a linear combination of the differientials of $g_1,\ldots,g_s$ at $p$.
%Thus, if $V_i\subseteq\mathsf{S}_{R,\CC}$, then
%$\Vert X\Vert_2^2-R\in{\bf I}(V_i)$ and hence for any $p\in V_i$, we have
%$p\in\bfV_\CC(\Delta_\bfg)$ and $R$ is a real zero of $I_{\bfg,Y}$ which
%is a contradiction. Therefore, if $V_i\cap\mathsf{S}_{R,\CC}\neq\emptyset$,
%then $V_i\not\subseteq\mathsf{S}_{R,\CC}$ and
%$\dim(V_i\cap\mathsf{S}_{R,\CC})=\dim(V_i)-1=0$
%by Krull's Principal Ideal Theorem \cite[Chap. V, Corollary 3.2]{kunz}.
%The conclusion follows.
\end{proof}

\begin{theorem}\label{th::R}
%Suppose there exists $R\in\RR$ such that Condition \ref{condition::R} holds,
Suppose that an ideal $I\subseteq\RR[X]$
%with $\dim(I)=1$
and a $\mathscr{R}\in\RR_+$ satisfy$:$
%that
\begin{inparaenum}[\upshape(i\upshape)]
\item $\dim(I)=1$$;$
\item $\mathscr{R}<\mathscr{R}_I$$;$
\item $\Gamma_\RR(f)\cap\mathsf{B}_{\mathscr{R}}
=\bfV_\RR(I)\cap\mathsf{B}_{\mathscr{R}}$,
\end{inparaenum}
then Condition \ref{con::curve} holds for $\mathscr{R}$.
%Hence, for any isolation radius $R$ with $R<\mathscr{R}_{\bfg}$,
%$R$ is a faithful radius.
%if there exists a positive $R\in\RR$ satisfying that
%$\crit_\RR(f)\cap \mathsf{B}_R=V_\RR(\Delta_\bfg)\cap
%\mathsf{B}_R=\{\bfz\}$, then $R$ is a faithful radius of $f$.
% if it satisfies Condition \ref{condition::R}.
%. For
%any $u\in \Gamma_\RR(f)\cap \mathsf{B}_R$,
%%there is a curve connecting $u$ and $0$, i.e.,
%there exists a differentiable function $\phi(t)=(\phi_1(t),\ldots,\phi_n(t))$ with
%$\phi(t)\subseteq \Gamma_\RR(f)$ for $0\le t\le 1$ such that $\phi(0)=0$ and
%$\phi(1)=u$, i.e. there is a piece of curve in $\Gamma_\RR(f)$ connecting $u$
%and $0$. Moreover, for any $0<\bar{t}<1$, we have
%\begin{equation}\label{eq::phi}
%\frac{d\sum_{i=1}^n\phi^2_i}{dt}(\bar{t})\neq 0.
%\end{equation}
\end{theorem}
\begin{proof}
%By Condition \ref{condition::R}, for any
Fix a $\bfz\neq u\in \Gamma_\RR(f)$ with $\Vert u\Vert_2<\mathscr{R}$.
Since $\mathscr{R}<\mathscr{R}_I$ and $\Gamma_\RR(f)\cap\mathsf{B}_{\mathscr{R}}
=\bfV_\RR(I)\cap\mathsf{B}_{\mathscr{R}}$, by Corollary \ref{cor::notiso} and
the definition of $\mathscr{R}_I$, we
have $\bfV_\RR(I)\cap\mathsf{B}_{\mathscr{R}}=\bfV_\RR(I^{(1)})\cap\mathsf{B}_{\mathscr{R}}$
and hence $u\in\bfV_\RR(I^{(1)})$.
%\cap \mathsf{B}_R\subseteq\bfV_\RR(\bfg)\cap\bfB_R$,
%by Condition \ref{condition::R}, there exists $1\le i_0\le n$ with
%$\frac{\partial{f}}{\partial{x_{i_0}}}(u)\neq 0$
%.  Then
%and $\Gamma_\RR(f)=V_{i_0}$ in a small neighborhood of $u$.
%Since the
%gradient vectors $\nabla g_{i_0,j}(u), 1\le j\le n, j\neq i_0$, are independent,
%by the implicit function theorem, $u$ lies in a curve $\mathcal{C}$ restricted
%in $\Gamma_\RR(f)\cap \mathsf{B}_R$. Now we prove $\mathcal{C}$ passes through the
%original, otherwise
%\[
%\min\ \Vert x\Vert_2^2\quad\text{s.t.}\ x\in \Gamma_\RR(f)\cap \mathsf{B}_R\cap\mathcal{C}.
%\]
%has a minimizer $v$ with $\Vert v\Vert_2\le R$ and a neighborhood $\mathcal{O}$ of
%$v$ such that $v$ is also a mininizer of
%\begin{equation}\label{eq::v}
%\min\ \Vert x\Vert_2^2\quad\text{s.t.}\ x\in V_{i_0}\cap\mathcal{O}.
%\end{equation}
%for some $1\le i_0\le n$. Since the LICQ condition is satisfied at $v$ for
%(\ref{eq::v}), KKT condition holds at $v$ which implies $M_{i_0}$ is singular.
%Therefore, there exists a differentiable function
%$\phi(t)=(\phi_1(t),\ldots,\phi_n(t))$ with $0\le t\le 1$ such that
%$\phi(0)=0$ and $\phi(1)=u$.
%
%Now suppose that there exists $0<\bar{t}<1$ so that
%\[
%\frac{d\sum_{i=1}^n\phi^2_i}{dt}(\bar{t})=0.
%\]
%Let $X=\phi(\bar{t})$ and without loss of generality, suppose that
%$\frac{\partial{f}}{\partial{x_{1}}}(X)\neq 0$. Then by
%By Condition \ref{condition::R}
%Moreover, the following matrix
Because $u\not\in \bfV_\RR(\Delta_{I^{(1)}})$,
%by Remark \ref{rmk::less},
there is a Jacobian matrix
%in $\Delta_{I^{(1)}}$
of the form
\begin{equation}\label{eq::J}
M=\left[\begin{array}{cccc}
\frac{\partial g_{i_1}}{\partial x_1}(u) & \cdots & \frac{\partial
g_{i_1}}{\partial x_n}(u)\\
\vdots & \vdots & \vdots\\
\frac{\partial g_{i_{n-1}}}{\partial x_1}(u) & \cdots &
\frac{\partial g_{i_{n-1}}}{\partial x_n}(u)\\ u_1 & \cdots & u_n
\end{array}\right]
\end{equation}
with full rank.

Let $\tilde{g}(X)=\sum_{i=1}^n u_iX_i-\sum_{i=1}^n
u_i^2$ and $\tilde{I}^{(1)}=\langle g_{i_1},\ldots,g_{i_{n-1}}\rangle$.
Define a function
$Y=G(X):=(g_{i_1}(X),\ldots,g_{i_{n-1}}(X),\tilde{g}(X))$,
%Since $u\in\Gamma_\RR(f)\subseteq\bfV_\RR(\bfg)$, we
then we have $G(u)=\bfz$ and the Jacobian of $G(X)$ at $u$ is nonsingular.
Hence, by the inverse function theorem, $G(X)$ is an invertible function
in a neighborhood $\mathcal{O}_u$ of $u$. Without loss of generality,
we can assume that $\mathcal{O}_u\subseteq\mathsf{B}_\mathscr{R}$.
Thus, an invertible funtion
$X=G^{-1}(Y)=(G^{-1}_1(Y),\ldots,G^{-1}_n(Y))$ exists in some neighborhood
$\mathcal{O}_\bfz$ of $\bfz$. Moreover, $X=G^{-1}(Y)$ is differentiable in
$\mathcal{O}_\bfz$.
%
%$w=G(X)$ defines
%a differentiable function .
%In the $w$-coordinate system and in the neighborhood $\mathcal{O}$,
%We have
%\[
%	\bfV_\RR(\tilde{I}^{(1)})\cap\mathcal{O}_u=\{G^{-1}(w)\mid
%	w_1=\cdots=w_{n-1}=0, w_n\in(a,b)\},
%\]
%for some interval $(a,b)$ with $0\in(a,b)$.
Define $\phi(t)=(\phi_i(t))$ with
$\phi_i(t)=G^{-1}_i(0,\ldots,0,t)$. Then, there is an interval $(a,b)$ such that
$\phi((a,b))=\bfV_\RR(\tilde{I}^{(1)})\cap\mathcal{O}_u$. Moreover, we have
$0\in(a,b)$ and $\phi(0)=u$.
%Hence there exists a one-to-one map
%$t=\varphi(w_n)$ such that $\phi(\varphi(w_n))=G^{-1}(0,\cdots,w_n)$ in the
%neighborhood $\mathcal{O}$.  Let $\varphi(\bar{w}_n)=\bar{t}$ and
%$\bar{w}=(0,\cdots,\bar{w}_n)$, then $X=\phi(\bar{t})=G^{-1}(\bar{w})$.
%Therefore,
%\[
%\frac{\partial G^{-1}_i}{\partial w_n}(\bar{w})=
%\frac{d \phi_i}{d t}(\bar{t})\frac{d\varphi}{dw_n}(\bar{w}_n),
%\quad\text{for}\ i=1,\cdots,n.
%\]
Then we have
\[
%\left[\begin{array}{cccc}
%\frac{\partial g_{i_1}}{\partial x_1}(u) & \cdots & \frac{\partial
%g_{i_1}}{\partial x_n}(u)\\
%\vdots & \vdots & \vdots\\
%\frac{\partial g_{i_{n-1}}}{\partial x_1}(u) & \cdots & \frac{\partial
%g_{i_{n-1}}}{\partial x_n}(u)\\
%u_1 & \cdots & u_n
%\end{array}\right]
M\cdot
\left[\begin{array}{c}
\frac{\partial G^{-1}_1}{\partial Y_n}(\bfz)\\
\vdots\\
\frac{\partial G^{-1}_{n-1}}{\partial Y_n}(\bfz)\\
\frac{\partial G^{-1}_n}{\partial Y_n}(\bfz)
\end{array}\right]
=
\left[\begin{array}{c}
\frac{\partial Y_1}{\partial Y_n}(\bfz)\\
\vdots\\
\frac{\partial Y_{n-1}}{\partial Y_n}(\bfz)\\
%\frac{d\sum_{i=1}^n\phi_i^2}{2dt}(\bar{t})\frac{d\varphi}{dw_n}(\bar{w}_n)
%\sum_{i=1}^n
\frac{\partial\sum_{i=1}^n (G^{-1}_i)^2}{2\partial Y_n}(\bfz)
\end{array}\right]=
\left[\begin{array}{c}
0\\
\vdots\\
0\\
\frac{\ud\sum_{i=1}^n\phi^2_i}{2\ud t}(0)
\end{array}\right]
\]
%Since $\phi(t)$ is regular at $\bar{t}$ (i.e., its derivative never vanishes at
%$\bar{t}$ or in words, $\phi(t)$ never slows to a stop or backtracks on itself
%at $\bar{t}$) and $\varphi$ is one-to-one,
By the implicit function theorem, there is an $i_0$ such that
%at least one
$\frac{\partial G^{-1}_{i_0}}{\partial Y_n}(\bfz)\neq 0$.
% which
%implies
Since the
%Jacobian
matrix $M$ in (\ref{eq::J}) is nonsingular,
%which is a contradiction.
it implies $\frac{\ud\sum_{i=1}^n\phi^2_i}{\ud t}(0)\neq 0$.
Recalling Condition \ref{con::curve}, since
$\Gamma_\RR(f)\cap\mathcal{O}_u
%=\bfV_\RR(I)\cap\mathcal{O}_u%
=\bfV_\RR(I^{(1)})\cap\mathcal{O}_u$,
it remains to prove that
$\bfV_\RR(I^{(1)})\cap\mathcal{O}_u=\bfV_\RR(\tilde{I}^{(1)})\cap\mathcal{O}_u$.

It suffices to prove that
$\bfV_\CC(I^{(1)})\cap\mathcal{O}=\bfV_\CC(\tilde{I}^{(1)})\cap\mathcal{O}$ for
some Zariski open set $\mathcal{O}\subseteq\CC^n$ containing $u$.
Let $\bfV_\CC(I^{(1)})=V_1\cup\cdots\cup V_s$ and
$\bfV_\CC(\td{I}^{(1)})=\wt{V}_1\cup\cdots\cup\wt{V}_t$ be the irreducible
decompositions of $\bfV_\CC(I^{(1)})$ and $\bfV_\CC(\td{I}^{(1)})$, respectively.
Since the first $n-1$ rows of $M$ is linear independent, there is a unique
irreducible component, say $\widetilde{V}_1$, of $\bfV_\CC(\tilde{I}^{(1)})$
containing $u$ and $\widetilde{V}_1$ is smooth of dimension one at $u$.
Let $V_1$ be an irreducible component of $\bfV_\CC(I^{(1)})$ containing $u$.
Since $I^{(1)}\supseteq\tilde{I}^{(1)}$, we have
$\bfV_\CC(I^{(1)})\subseteq\bfV_\CC(\tilde{I}^{(1)})$ and hence
$V_1\subseteq\widetilde{V}_1$. Because $\dim(V_1)=1$, it follows that
$V_1=\widetilde{V}_1$ which also implies that $V_1$ is the unique irreducible
component of $\bfV_\CC(I^{(1)})$ containing $u$.
Let $\mathcal{O}=\CC^n\backslash\left(\bigcup_{i=2}^s V_i\cup\bigcup_{i=2}^t\wt{V}_i\right)$,
then we have
$\bfV_\CC(I^{(1)})\cap\mathcal{O}=\bfV_\CC(\tilde{I}^{(1)})\cap\mathcal{O}$
which ends the proof.
\end{proof}

Combining Theorems \ref{th::main2} and \ref{th::R}, we obtain
\begin{theorem}\label{th::fr}
Suppose that an ideal $I\subseteq\RR[X]$
and a $\mathscr{R}\in\RR_+$ satisfy$:$
\begin{inparaenum}[\upshape(i\upshape)]
\item $\dim(I)=1$$;$
\item $\mathscr{R}<\mathscr{R}_I$$;$
\item $\Gamma_\RR(f)\cap\mathsf{B}_{\mathscr{R}}
=\bfV_\RR(I)\cap\mathsf{B}_{\mathscr{R}}$,
\end{inparaenum}
%	Suppose that an ideal $I\subseteq\RR[X]$ with $\dim(I)=1$
%	and an isolation radius $R$ of $\bfz$
%	satisfys that $R\le\mathscr{R}_I$ and
%	$\Gamma_\RR(f)\cap\mathsf{B}_R=\bfV_\RR(I)\cap\mathsf{B}_R$,
	%for an isolation radius $R$ of $f$,
	then any isolation radius $R\in\RR_+$ of $\bfz$
	with $R<\mathscr{R}$ is a faithful radius of $\bfz$.
\end{theorem}

Let $\boldsymbol{\gamma}=\{\gamma_{i,j}\mid i,j=1,\ldots,n\}$ where
$\gamma_{i,j}:=\frac{\partial f}{\partial X_i}X_j-\frac{\partial f}{\partial X_j}X_i$,
then $\Gamma_\RR(f)=\bfV_\RR(\langle\boldsymbol{\gamma}\rangle)$.
If $\dim(\langle\boldsymbol{\gamma}\rangle)=1$,
let $\mathscr{R}_{\langle\boldsymbol{\gamma}\rangle}$ be defined as in (\ref{eq::g}).

\begin{cor}
If $\dim(\langle\boldsymbol{\gamma}\rangle)=1$, then any isolation radius
$R$ with $R<\mathscr{R}_{\langle\boldsymbol{\gamma}\rangle}$ is a faithful
radius of $\bfz$.
\end{cor}

Note that the ideal $\langle\boldsymbol{\gamma}\rangle$ may not be one-dimensional. Let
$\mathcal{G}:={\bf I}\left({\overline{\Gamma_\CC(f)\backslash\crit_\CC(f)}}^{\mathcal{Z}}\right)$,
i.e., the vanishing ideal of ${\overline{\Gamma_\CC(f)\backslash\crit_\CC(f)}}^{\mathcal{Z}}$
in $\RR[X]$.
According to Corollary \ref{cor::dim1}, we have $\dim(\mathcal{G})=1$
up to a generic linear change of coordinates. Moreover,
\begin{prop}\label{prop::r}
% If $\mathscr{R}_\bfg>0$,
%$R$ is a faithful radius.
$\Gamma_\RR(f)\cap\mathsf{B}_R=\bfV_\RR(\mathcal{G})\cap\mathsf{B}_R$ holds
for any isolation radius $R$.
%with $\crit_\RR(f)\cap\mathsf{B}_R=\{\bfz\}$,
\end{prop}
\begin{proof}
%$\bfV_\CC(\Delta)$ contains the singular locus of
%${\overline{\Gamma_\CC(f)\backslash\crit_\CC(f)}}^{\mathcal{Z}}$
%and the set of critical points of the mapping $\Vert X\Vert_2^2:
%{\overline{\Gamma_\CC(f)\backslash\crit_\CC(f)}}^{\mathcal{Z}}\rightarrow \CC$.
%Up to a generic linear change of coordinates, the set
%${\overline{\Gamma_\CC(f)\backslash\crit_\CC(f)}}^{\mathcal{Z}}$ is one-dimensional,
%hence its singular locus is of dimension zero.
%Additionally, there are finitely many critical values of the mapping $\Vert X\Vert_2^2$
%by Sard's theorem. Therefore,
%we can find a $R\in\RR$ such that $\bfV_\RR(\Delta)\cap {\bf B}_R=\{\bfz\}$.
%By Corollary \ref{cor::dim1} and Proposition \ref{prop::cs}, we have
%$\mathscr{R}_\bfg>0$. We now need to show that
%$\Gamma_\RR(f)\cap\mathsf{B}_R\backslash\{\bfz\} \subseteq\bfV_\RR(\bfg)\cap\mathsf{B}_R$.
Let $V_1\cup\cdots\cup V_s\cup V_{s+1}\cup\cdots\cup V_{t}$
be the decomposition of $\Gamma_\CC(f)$ as a union of irreducible components.
Assume that $V_i\not\subseteq\crit_\CC(f)$ for $1\le i\le s$ and
$V_j\subseteq\crit_\CC(f)$ for $s+1\le j\le t$.
Let $V^{(1)}=V_1\cup\cdots\cup V_s$ and $V^{(2)}=V_{s+1}\cup\cdots\cup V_{t}$,
then $\bfV_\CC(\mathcal{G})=V^{(1)}$ and $V^{(2)}\cap\mathsf{B}_R\subseteq\{\bfz\}$.
We have $\Gamma_\RR(f)\cap\mathsf{B}_R=\Gamma_\CC(f)\cap\mathsf{B}_R
=(V^{(1)}\cap\mathsf{B}_R)\cup (V^{(2)}\cap\mathsf{B}_R)\subseteq
(\bfV_\RR(\mathcal{G})\cap\mathsf{B}_R)\cup\{\bfz\}$.
%By Theorem \ref{th::main2} and \ref{th::R},
%For any $0<\varepsilon<R$, recall the definition of $f^{\max}_r$ and
%$f^{\min}_r$ in (\ref{eq::ops}).
%Then, at least one of $f^{\max}_r$ and $f^{\min}_r$ is nonzero; otherwise,
%$f$ is a zero polynomial.
%Hence by Proposition \ref{prop::uv}, there exists a nonzero
%$u\in\Gamma_\RR(f)\cap\mathsf{B}_\varepsilon$ for any $0<\varepsilon<R$.
%%with $\Vert u\Vert_2\le\varepsilon$.
%Since $R$ is an isolation radius and
Since $\Gamma_\RR(f)\backslash\crit_\RR(f)\subset\bfV_\RR(\mathcal{G})$,
by Corollary \ref{cor::notiso},
$\bfz\in\bfV_\RR(\mathcal{G})$ and hence
$\Gamma_\RR(f)\cap\mathsf{B}_R\subseteq\bfV_\RR(\mathcal{G})\cap\mathsf{B}_R$.
It is clear that
$\Gamma_\RR(f)\cap\mathsf{B}_R\supseteq\bfV_\RR(\mathcal{G})\cap\mathsf{B}_R$
and thus the conclusion follows.
\end{proof}
\begin{cor}\label{cor::fr}
	Suppose that $\dim(\mathcal{G})=1$ and $R$ is an isolation radius of
	$f$. Then, any $r\in\RR_+$ with $r<\min\{R,\mathscr{R}_\mathcal{G}\}$
	is a faithful radius of $\bfz$.
\end{cor}

%Recall that
%$\mathcal{G}={\bf I}\left({\overline{\Gamma_\CC(f)\backslash\crit_\CC(f)}}^{\mathcal{Z}}\right)$.

Provided that an isolation radius of $\bfz$ is known,
we now present an algorithm to compute a faithful radius of $\bfz$.
\begin{framed}
%\hrule
\begin{algorithm}\label{al::fr}{\sf
\noindent FaithfulRadius($f,R_{\text{\sf iso}}$)\\
Input: A polynomial $f\in\RR[X]$ with $\bfz$ as an isolated real
critical point and an isolation radius $R_{\text{\sf iso}}$ of $\bfz$. \\
Output: $R\in\RR_+$ such that any $0<r<R$ is a faithful radius of $\bfz$.
\begin{enumerate}[\upshape 1.]
	\item If
		%the ideal $\langle\boldsymbol{\gamma}\rangle$ is radical and
		$\dim(\langle\boldsymbol{\gamma}\rangle)=1$,
		then let $I=\langle\boldsymbol{\gamma}\rangle$;
		otherwise, make a linear change
		of coordinates of $f$ such that $\dim(\mathcal{G})=1$ and let
		$I=\mathcal{G}$;
	\item Compute the equidimensional decomposition $I=I^{(0)}\cap I^{(1)}$
	%the radical ideal $\sqrt{I^{(1)}}$ and let
		and the set $\Delta_{I^{(1)}}$ as defined in (\ref{eq::g});
	%$=\langle g_1,\ldots,g_s\rangle$.
		%Let $\Delta_{I^{(1)}}=I+\langle\text{all $n$-minors of}\
	%\jac\left(g_1,\ldots,g_s,\Vert X\Vert_2^2\right)\rangle$;
	\item Compute elimination ideals
		$I_{{n}}^{(0)}:=(I^{(0)}+\langle\Vert X\Vert_2^2-X_{n+1}\rangle)\cap\RR[X_{n+1}]$
		and $I_{n}^{(1)}:=
		(\langle\Delta_{I^{(1)}}\rangle+\langle\Vert X\Vert_2^2-X_{n+1}\rangle)\cap\RR[X_{n+1}]$;
	\item Compute the isolation intervals $\{[a_i,b_i]\mid i=1,\ldots,t\}$ of
		$\bfV_\RR(I_{n}^{(0)}\cdot I_{n}^{(1)})$;
	\item Let $\mathscr{R}=\min\{\sqrt{a_i}\mid
	%0\not\in[a_i,b_i],
	a_i>0, i=1,\ldots,t\}$ and return $R=\min\{\mathscr{R},R_{\text{\sf iso}}\}$.
\end{enumerate}}
\end{algorithm}
%\hrule
\end{framed}

\begin{theorem}
Algorithm \ref{al::fr} runs successfully and is correct.
\end{theorem}
\begin{proof}
	According to the proof of Proposition \ref{prop::r}, we have
	$\bfz\in \bfV_\RR(I^{(1)})$
	and hence $0\in\bfV_\RR(I_{n}^{(0)}\cdot I_{n}^{(1)})$ by the definition of
	$\Delta_{I^{(1)}}$.
	Then, we have $\mathscr{R}<\mathscr{R}_I$ in Step 5 since $[a_i,b_i]$'s are isolation intervals of
	$\bfV_\RR(I_{n}^{(0)}\cdot I_{n}^{(1)})$.
	Then, by Corollary \ref{cor::dim1} and the proof of Lemma \ref{lem::cs},
	the algorithm runs successfully. Its correctness can be seen by
	combining Theorem \ref{th::fr} and Corollary \ref{cor::fr}.
\end{proof}

\begin{example}\label{ex::ex1}
Consider the polynomial $f=X_1^2+(1-X_1)X_2^4$	discussed in the
introduction.
The origin $\bfz$ is an isolated real critical point and degenerate.
The graphs of $f$ are shown in Figure~\ref{graph_ex1}. On the left hand side,
the graph is drawn with the variables $X_1, X_2$ varying in the range $[-2,2]$.
It seems from this graph that $\bfz$ is a saddle point. However,
if we zoom in, then we get the graph on the right hand side which indicates
that $\bfz$ is in fact a strict local minimizer.
Now we use Algorithm \ref{al::fr} to obtain a faithful radius of $\bfz$.
\begin{figure}
	\centering
	\caption{\label{graph_ex1} The graphs of $f$ in Example \ref{ex::ex1}.}
\scalebox{0.5}{
\includegraphics{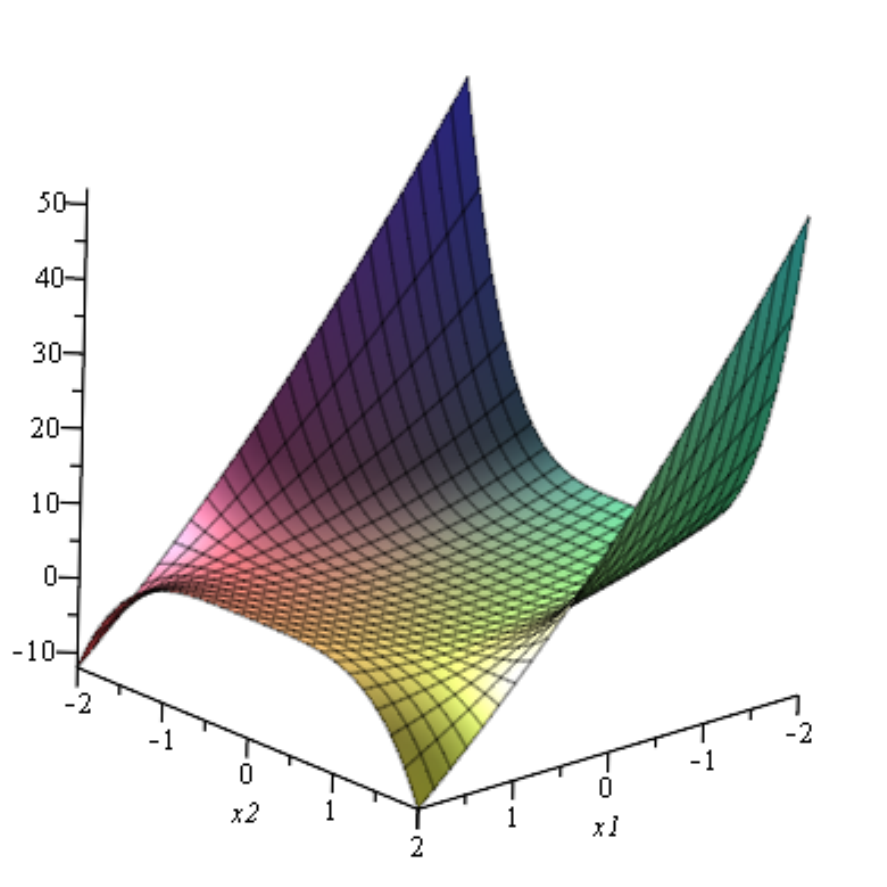}
\qquad\qquad
\includegraphics{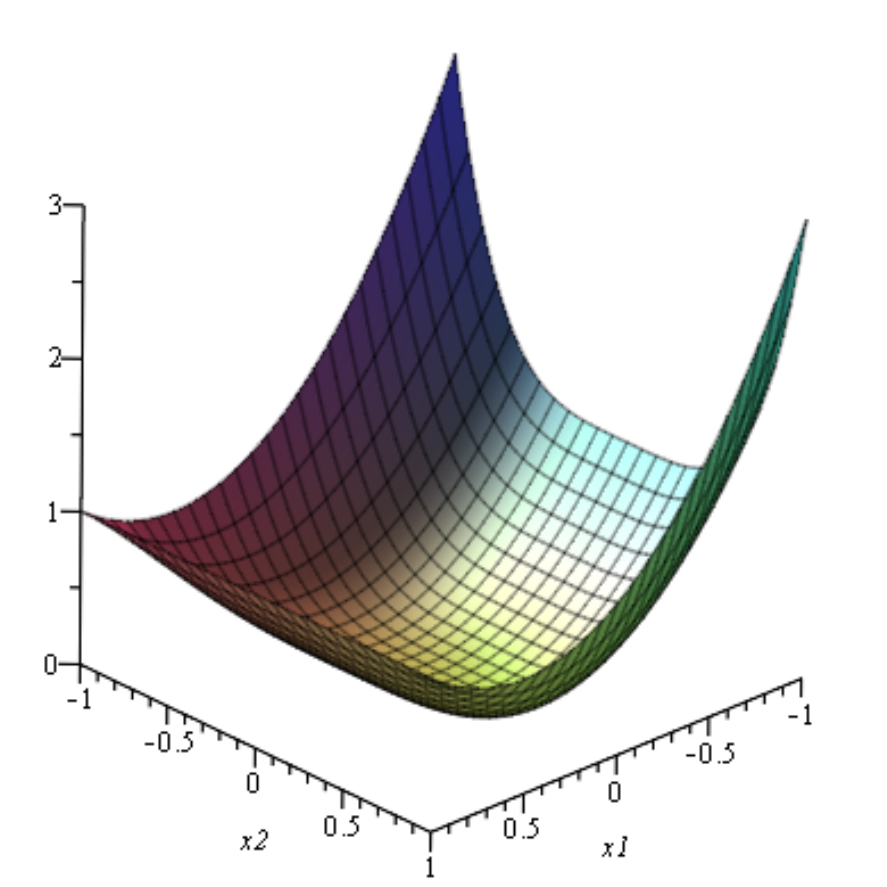}}
\end{figure}

It is easy to check that $R_{\text{\sf iso}}=1$ is an isolation radius of $\bfz$ and
$\boldsymbol{\gamma}=\{4X_1^2X_2^3-X_2^5-4X_1X_2^3+2X_1X_2\}$.
The curve of tangency $\Gamma_\RR(f)=\bfV_\RR(\langle\boldsymbol{\gamma}\rangle)$ is shown (red)
in Figure~\ref{curve_ex1}.
Since
%the ideal $\langle\boldsymbol{\gamma}\rangle$ is radical and
$\dim(\langle\boldsymbol{\gamma}\rangle)=1$,
we let $I=\langle\boldsymbol{\gamma}\rangle$.
We implement Algorithm \ref{al::fr} in the software Maple.
With inputs $f$ and $R_{\text{\sf iso}}$, we get the return
$R=\frac{\sqrt{767451466998008631606300139861}}{1125899906842624}
\approx 0.778<1$.
The circle centered at $\bfz$ with radius $R$ is shown (blue)
in Figure \ref{curve_ex1}.
Hence, any $r<R$ is a faithful radius of $\bfz$.
\begin{figure}
	\centering
	\caption{\label{curve_ex1} The curve of tangency of $f$ in Example \ref{ex::ex1}.}
\scalebox{0.5}{
	\includegraphics{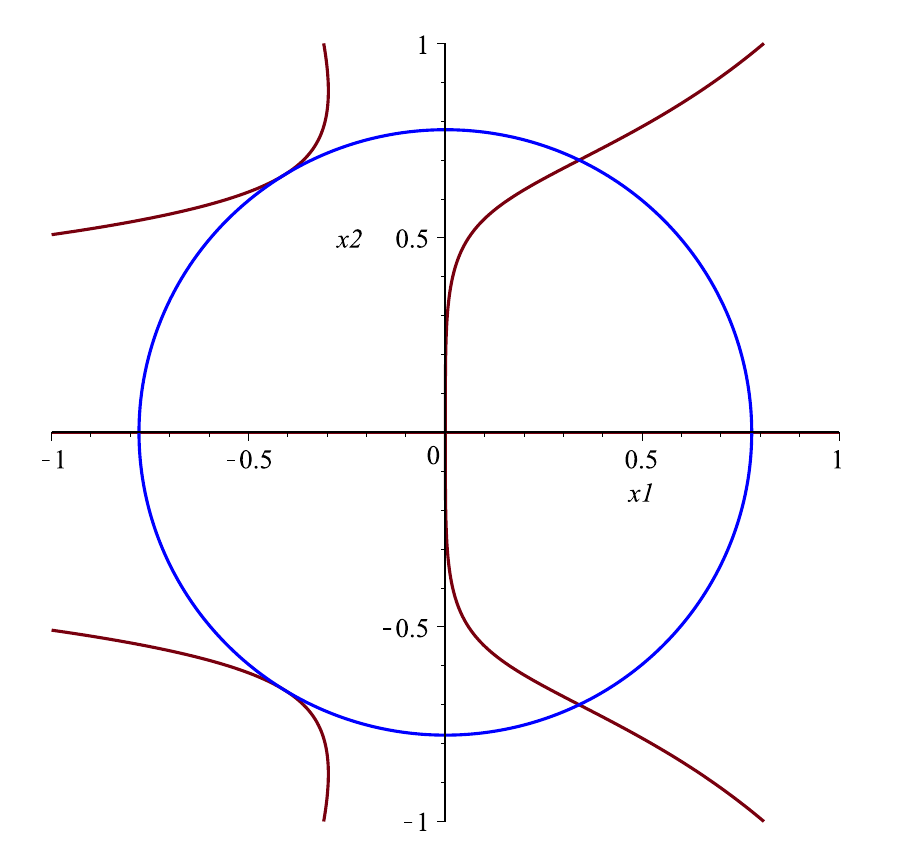}}
\end{figure}
\end{example}

\subsection{On the computation of isolation radius}

%Recall that $\Vert X\Vert_2^2=X_1^2+\cdots+X_n^2$.
As we have seen, if an isolation radius of $\bfz$ is available,
a faithful radius of $\bfz$ can be obtained by Algorithm \ref{al::fr}.
To end this section, we propose some strategies to compute an isolation radius of $\bfz$.

Let $\mathscr{C}=
\langle\frac{\partial f}{\partial X_1},\ldots,\frac{\partial f}{\partial X_n}\rangle$.
If $\dim(\mathscr{C})=0$, then an isolation radius of $\bfz$ can be computed by the
RUR method \cite{Rouillier1999} for zero-dimensional systems.

%Let $\mathscr{C}=
%\langle\frac{\partial f}{\partial X_1},\ldots,\frac{\partial f}{\partial X_n}\rangle$
Assume that $\dim(\mathscr{C})>0$.  We now borrow the idea from \cite{AUBRY2002543} which aims
to compute one point on each semi-algebraically connected component of a real algebraic
variety by the critical point method of a distance function.
Compute the equidimensional decomposition
$\mathscr{C}=\mathscr{C}^{(0)}\cap \mathscr{C}^{(1)}\cap\cdots\cap \mathscr{C}^{(t)}$
where $\mathscr{C}^{(k)}$ is radical and of dimension $k$ for each $k=0,\ldots,t$.
For an efficient algorithm of such decomposition, see \cite[Section 3]{AUBRY2002543}.
Then, we can compute the minimal distance $d^{(k)}$
of $\bfV_\RR(\mathscr{C}^{(k)})\backslash\{\bfz\}$ to
$\bfz$ for each $k$ and choose a positive number less than
the smallest one as an isolation radius.
%For each equidimensional component $I$ of $\mathscr{C}$, consider the following procedure.
Suppose that $\mathscr{C}^{(k)}=\langle h_1,\ldots,h_l\rangle$.
%with
%$\dim(I)=d$, let $\sig(I)\subset\CC^n$ be the set of singular points of $\bfV_\CC(I)$,
%i.e., those points at which the Jacobian matrix $\jac(g_1,\ldots,g_s)$ is of rank $<n-d$.
Let $\mathscr{M}(\mathscr{C}^{(k)})$ be the set of $h_1,\ldots,h_l$ and all the
$(n-k+1,n-k+1)$ minors of the Jacobian matrix
$\jac(h_{i_1},\ldots,h_{i_{n-k}},\Vert X\Vert_2^2)$
for all $\{i_1,\ldots,i_{n-k}\}\subset\{1,\ldots,l\}$.
Consider the map 	
\[
\begin{aligned}
	\Phi:\ \bfV_\CC(\mathscr{C}^{(k)})&\ \rightarrow\ &&\CC\\
	x&\ \mapsto \ &&x_1^2+\cdots+x_n^2.
\end{aligned}
\]
Then, $\bfV_\CC(\mathscr{M}(\mathscr{C}^{(k)}))$
consists of the singular locus $\sing(\mathscr{C}^{(k)})$ of
$\bfV_\CC(\mathscr{C}^{(k)})$ and the set of critical points of $\Phi$.
By the first part in the proof of \cite[Theorem 2.3]{AUBRY2002543},
the point of $\bfV_\RR(\mathscr{C}^{(k)})\backslash\{\bfz\}$ at the minimal distance to
$\bfz$ is contained in $\bfV_\RR(\mathscr{M}(\mathscr{C}^{(k)}))$.
Compute the elimination ideal
$\mathscr{M}_n(\mathscr{C}^{(k)})=\langle\mathscr{M}(\mathscr{C}^{(k)}),
\Vert X\Vert_2^2-X_{n+1}\rangle\cap\RR[X_{n+1}]$ and then we have
$(d^{(k)})^2\in\bfV_\RR(\mathscr{M}_n(\mathscr{C}^{(k)}))$. If
$\mathscr{M}_n(\mathscr{C}^{(k)})\neq\langle 0\rangle$,
then the smallest positive real root in $\bfV_\RR(\mathscr{M}_n(\mathscr{C}^{(k)}))$,
which can be obtained by any real root isolation algorithm for univariate polynomials,
is a lower bound of $(d^{(k)})^2$.
If $\mathscr{M}_n(\mathscr{C}^{(k)})=\langle 0\rangle$, by Sard's theorem, it happens
if and only if
the set $\Phi(\sing(\mathscr{C}^{(k)}))$ is infinite. In this case, we can replace
$\mathscr{C}$ by $\mathscr{M}(\mathscr{C}^{(k)})$ and repeat the above procedure recursively.
Since $\dim(\sing(\mathscr{C}^{(k)}))<k$, this process will finitely terminate and return
an isolation radius of $\bfz$.

Alternatively, when $\dim(\mathscr{C})>0$, we can compute an isolation radius of $\bfz$
by testing the emptyness of a real algebraic variety. Adding two new variables $X_{n+1}$
and $X_{n+2}$, a $R\in\RR_+$ is an isolation radius of $\bfz$ if and only if the following
polynomial system has no real root
\[
	\frac{\partial f}{\partial X_1},\ldots,\frac{\partial f}{\partial X_n},
	\Vert X\Vert_2^2+X_{n+1}^2-R^2, \Vert X\Vert_2^2\cdot X_{n+2}-1.
\]
Hence, we can set an initial $R$ and test the emptyness of the real algebraic variety
generated by the above polynomials. If it is empty, then $R$ is an isolation radius;
otherwise, try $R/2$ and repeat. For algorithms of such tests, see
\cite[Section 4]{AUBRY2002543} and \cite{SafeyElDin2003}.

%\begin{framed}
%\begin{algorithm}\label{al::fr}{\sf
%\noindent faithfulRadius($f,R,I$)\\
%Input: A polynomial $f\in\RR[X]$ with $\bfz$ as a degenerate critical point.
%An isolation radius $R\in\RR_+$ of $f$. A radical ideal $I\subseteq\RR[X]$ such that
%$\dim(I)=1$ and $\Gamma_\RR(f)\cap\mathsf{B}_R =\bfV_\RR(I)\cap\mathsf{B}_R$\\
%%such that
%%$\langle\bfg\rangle$ is radical, equidimensional one and
%%$\Gamma_\RR(f)\cap\mathsf{B}_R\backslash\{\bfz\}=\bfV_\RR(\bfg)\cap\mathsf{B}_R$\\
%Output: a faithful radius of $\bfz$ as a degenerate critical point of $f$
%\begin{enumerate}[\upshape 1.]
%	\item Compute the ideal $I_{\Delta_\bfg,Y}\subseteq\RR[Y]$;
%	\item
%	%If $I_{\Delta_\bfg,Y}\neq\emptyset$,
%		Compute a lower bound $r>0$ of $\mathscr{R}_\bfg$;
%%          of the minimal nonzero real root of $I_{\Delta_\bfg,Y}$;
%  \item Fix a $R_1<\min\{R,r\}$ and
%	  $I_{\bfg,Y}=\langle\bfg,\Vert X\Vert_2^2-R_1,f-Y\rangle\subseteq\RR[X,Y]$;
%  \item Compute intervals $\{[a_i,b_i]\mid i=1,\ldots,s\}$ such that
%	  $0\not\in[a_i,b_i]$ for each $i$ and all coordinates $Y$
%	  of the real roots of $I_{\bfg,Y}$ lie in some unique $[a_i,b_i]$;
%  \item Let $m=\min\{a_i, i=1,\ldots,s\}$ and $M=\max\{b_i,i=1,\ldots,s\}$;
%  \item If $m>0$, return ``local minimizer''; if $M<0$, return
%	  ``local maximizer''; if $m<0<M$, return ``saddle point''.
%\end{enumerate}}
%\end{algorithm}
%\end{framed}

\section{Certificates of types of degenerate critical points}\label{sec::types}

If $R\in\RR_+$ is a faithful radius of the isolated real critical point
$\bfz$ of $f$, then we can compute the extrema $f_R^{\min}$ and $f_R^{\max}$
in $(\ref{eq::ops})$ to classify the type of $\bfz$ by Theorem \ref{th::main}.
%However, if $f_R^{\min}=0$
%or $f_R^{\max}=0$, numerical errors in the output of any approximation methods for
%the optimization problems in $(\ref{eq::ops})$ may mislead us to claim that $\bfz$
%is a saddle point.
To deal with the issues when computing $f_R^{\min}$ and $f_R^{\max}$ as mentioned
in the introduction,
we next show that how to decide the type of $\bfz$ by means of real root
isolation of zero-dimensional polynomial systems. Recall the notation
$\mathsf{S}_r$ in (\ref{eq::BS}).
%\subsection{}
\begin{prop}\label{prop::connect2S}
	Suppose that $\mathscr{R}\in\RR_+$ satisfies Condition \ref{con::curve}
	and $0<R<\mathscr{R}$.
	Then for any $\bfz\neq u\in\Gamma_\RR(f)\cap\mathsf{B}_R$,
	there exists a continuous map
	$\varphi(t): [a,b]\rightarrow\Gamma_\RR(f)\cap\mathsf{B}_R$ with
	$\varphi(a)=u$, $\varphi(b)\in\mathsf{S}_R$ and $\bfz\not\in\varphi([a,b])$.
\end{prop}
\begin{proof}
Consider the following semi-algebraic set
\[
S:=\Gamma_\RR(f)\cap\{x\in\RR^n\mid \Vert u\Vert_2^2/2\le \Vert x\Vert_2^2\le R^2\}.
\]
Let $\mathcal{C}$ be the connected component of $S$ containing $u$.
If $\mathcal{C}\cap\mathsf{S}_R\neq\emptyset$, then the conclusion
follows since $\mathcal{C}$ is path connected.
Otherwise, the function $\Vert X\Vert_2^2$ reaches its maximum on
$\cal{C}$ at a maximizer in $\cal{C}$. Then we can get a contradiction
%By the assumption, there exist
%Otherwise, there are two cases.
%Case 1: $\mathcal{C}=\{u\}$,
%then $u$ is a local maximizer of the function $\Vert X\Vert_2^2$
%on $\mathcal{C}'$;
%Case 2: $\mathcal{C}\neq\{u\}$, then
%$\Vert X\Vert_2^2$ attains a local maximizer $v$ on $\mathcal{C}$.
%In both cases, we can obtain the conclusion by
using arguments similar to
the first part of the proof of Theorem \ref{th::main2}.
\end{proof}

For any $r\in\RR_+$, comparing with Corollary \ref{cor::uv}, define
\begin{equation}\label{eq::fbar}
%\begin{aligned}
f_r^{-}:=\min\{f(x)\mid x\in \Gamma_\RR(f)\cap \mathsf{S}_r\}
\quad\text{and}\quad
f_r^{+}:=\max\{f(x)\mid x\in \Gamma_\RR(f)\cap \mathsf{S}_r\}.
%\end{aligned}
\end{equation}
\begin{theorem}\label{th::con}
	Suppose that $\mathscr{R}\in\RR_+$ satisfies Condition \ref{con::curve}.
	Then for any isolation radius $R<\mathscr{R}$,
	%with $\crit_\RR(f)\cap\mathsf{B}_R=\{\bfz\}$,
	it holds that
\begin{enumerate}[\upshape (i)]
	\item $\bfz$ is a local minimizer if and only if $f_R^{-}>0$;
\item $\bfz$ is a local maximizer if and only if $f_R^{+}<0$;
\item $\bfz$ is a saddle point if and only if $f_R^{+}>0>f_R^{-}$.
\end{enumerate}
\end{theorem}
\begin{proof}
	By Theorem \ref{th::main2}, $R$ is a faithful radius of $\bfz$.
	According to Theorem \ref{th::main} and Definition \ref{def::faithr} (ii),
	the ``only if'' parts in (i), (ii) and the ``if'' part in (iii) are clear.
	%It only needs to prove the ``if'' parts (i) and (ii), then (iii) follows.
	
	(i). ``if'' part. Suppose that $f_R^->0$, then we have $f^{\min}_R=0$.
	Otherwise, by Corollary \ref{cor::uv}, there exists
	$\bfz\neq u\in\Gamma_\RR(f)\cap\mathsf{B}_R\backslash\mathsf{S}_R$ such
	that $f(u)<0$. By Proposition \ref{prop::connect2S},
	there exists a continuous map
	$\varphi(t): [a,b]\rightarrow\Gamma_\RR(f)\cap\mathsf{B}_R$ with
	$\varphi(a)=u$ and $\varphi(b)\in\mathsf{S}_R$. Then we have a continuous
	function $g(t):=f(\varphi(t)):[a,b]\rightarrow \RR$ such that
	$g(a)<0$ and $g(b)>0$. By the mean value theorem, there exists
	$\bar{t}\in(a,b)$ such that $g(\bar{t})=f(\varphi(\bar{t}))=0$.
	Since that $\bfz\neq\varphi(\bar{t})\in\Gamma_\RR(f)\cap\mathsf{B}_R$ by
	Proposition \ref{prop::connect2S} and $R$ is a faithful radius,
	we get a contradiction;

	Similarly, we can prove (ii) and then (iii) follows.
%	if (i) and (ii)
%	are correct.
\end{proof}

%
%Thus, $\mathscr{R}$ satisfies Condition \ref{con::curve}.
%By Theorem \ref{th::main2}, $R$ is faithful radius.
%Recall the definition of $\Delta_\bfg$ and let
%$I_{\Delta_\bfg,Y}:=\langle\Delta_\bfg, \Vert X\Vert_2^2-Y\rangle\cap\RR[Y]$.
For any $r\in\RR_+$, let $\mathsf{S}_{r,\CC}=\{x\in\CC^n\mid\sum_{i=1}^n x_i^2=r^2\}$.
Recall the definition $\mathscr{R}_I$ for an ideal $I$ in (\ref{eq::g}).
\begin{prop}\label{prop::cs}
%Suppose that the ideal $I$ is radical and
	Given an ideal $I\subseteq\RR[X]$ with $\dim(I)=1$,
%Let $I_Y=\langle\Delta_\bfg, \Vert X\Vert_2^2-Y\rangle\cap\RR[Y]$,
%Then
%\begin{enumerate}[\upshape (i)]
%	\item $I_{\Delta_\bfg,Y}\neq\emptyset$ and hence $\mathscr{R}_\bfg>0$.
	%Let $0<r_1<\cdots<r_s$ be the real zeros of $I_{\Delta_\bfg,Y}$,
%then $r_1\le\mathscr{R}_\bfg$;
%	\item
the system $\bfV_\CC(I)\cap\mathsf{S}_{R,\CC}$ is zero-dimensional
for any $0<R<\mathscr{R}_I$,
%\end{enumerate}
\end{prop}
\begin{proof}
%(i). Consider the map 	
%\[
%\begin{aligned}
%	F:\ \bfV_\CC(\langle\bfg\rangle)&\ \rightarrow\ &&\CC\\
%	x&\ \rightarrow\ &&x_1^2+\cdots+x_n^2.
%\end{aligned}
%\]
%By the assumption, $\bfV_\CC(\Delta_\bfg)$ consists of the singular locus of
%$\bfV_\CC(\langle\bfg\rangle)$ and the set of critical points of $F$.
%Since $\dim(\bfV_\CC(\langle\bfg\rangle))=1$, its singular locus is zero-dimensional.
%By the algebraic Sard's theorem, there are only finitely many critical values
%of the map $F$ (note that if $F$ is not dominant, the conclusion is clearly true).
%Hence, the set $F(\Delta_\bfg)$ is finite and $I_{\Delta_\bfg,Y}\neq\emptyset$
%by the Closure Theorem \cite[Chap. 3, \S 2, Theorem 3]{CLO}.
%Clearly, we have $r_1\le\mathscr{R}_\bfg$.
It only needs to prove that $\bfV_\CC(I^{(1)})\cap\mathsf{S}_{R,\CC}$ is
zero-dimensional.
Let $V_1\cup\cdots\cup V_s$ be the decomposition of
$\bfV_\CC(I^{(1)})$ as a union of irreducible components.
Fix an $1\le i\le s$.
If $V_i\cap\mathsf{S}_{R,\CC}\neq\emptyset$,
then we show that $V_i\not\subseteq\mathsf{S}_{R,\CC}$.
To the contrary, assume that $V_i\subseteq\mathsf{S}_{R,\CC}$.
By the definition of $\mathscr{R}_I$ and the proof of Lemma \ref{lem::cs},
we have $V_i\cap V_j=\emptyset$ for each $j\neq i$
since $V_i\cap V_j$ is contained in the singular locus of
$\bfV_\CC(I^{(1)})$.
Then for any point $p$ in the nonsingular part of $V_i$, it holds that
$T_p(V_i)=T_p(\bfV_\CC(I^{(1)}))$ where $T_p(V_i)$
and $T_p(\bfV_\CC(I^{(1)}))$ denote the tangent spaces of $V_i$ and
$\bfV_\CC(I^{(1)})$ at $p$, respectively.
Then,
for any $h\in{\bf I}(V_i)$, the differiential of $h$ at $p$ can be expressed as
a linear combination of the differientials of $g_1,\ldots,g_s$ (the
generators of $\sqrt{I^{(1)}}$) at $p$.
In particular, by the assumption that $V_i\subseteq\mathsf{S}_{R,\CC}$, it holds for
$h:=\Vert X\Vert_2^2-R^2\in{\bf I}(V_i)$.
Hence, all the determinants in the set $\mathscr{D}$ in (\ref{eq::g}) vanish at any $p\in V_i$
since $\dim(I^{(1)})=1$.
Consequently, we have
$p\in\bfV_\CC(\Delta_{I^{(1)}})$. By the definition,
it implies that $R\ge\mathscr{R}_I$ which
is a contradiction. Therefore, if $V_i\cap\mathsf{S}_{R,\CC}\neq\emptyset$,
then $V_i\not\subseteq\mathsf{S}_{R,\CC}$ and
$\dim(V_i\cap\mathsf{S}_{R,\CC})=\dim(V_i)-1=0$
by Krull's Principal Ideal Theorem \cite[Chap. V, Corollary 3.2]{kunz}.
The conclusion follows.
\end{proof}

Recall the definition of $\boldsymbol{\gamma}$ and Algorithm \ref{al::fr}.
We now give an algorithm to decide the type of the isolated real critical point $\bfz$
of $f$.

\begin{framed}
%\hrule
\begin{algorithm}\label{al::main}{\sf
\noindent Type($f,R_{\text{\sf iso}}$)\\
Input: A polynomial $f\in\RR[X]$ with $\bfz$ as an isolated real
critical point and an isolation radius $R_{\text{\sf iso}}$ of $\bfz$. \\
%such that
%$\langle\bfg\rangle$ is radical, equidimensional one and
%$\Gamma_\RR(f)\cap\mathsf{B}_R\backslash\{\bfz\}=\bfV_\RR(\bfg)\cap\mathsf{B}_R$\\
Output: The type of $\bfz$ as a critical point of $f$.
\begin{enumerate}[\upshape 1.]
	\item If $\dim(\langle\boldsymbol{\gamma}\rangle)=1$,
		then let $I=\langle\boldsymbol{\gamma}\rangle$;
		otherwise, make a linear change
		of coordinates of $f$ such that $\dim(\mathcal{G})=1$ and let
		$I=\mathcal{G}$;
%	\item If the ideal $\langle\bfg\rangle$ is radical and $\dim(\langle\bfg\rangle)=1$,
%		then let $I=\langle\bfg\rangle$; otherwise, make a linear change
%		of coordinates of $f$ such that $\dim(\mathcal{G})=1$ and let
%		$I=\mathcal{G}$.
%	\item Compute $\mathscr{R}_I$ and a $R_1\in\RR_+$
%		such that $R_1<\min\{R,\mathscr{R}_I\}$.
	%If $I_{\Delta_\bfg,Y}\neq\emptyset$,
	%	Compute a lower bound $r>0$ of $\mathscr{R}_\bfg$;
%            of the minimal nonzero real root of $I_{\Delta_\bfg,Y}$;
	\item Let $R=\text{\sf FaithfulRadius}(f,R_{\text{\sf iso}})$ and fix a
		radius $0<r<R$;
	\item Let $\bar{I}=I+\langle\Vert X\Vert_2^2-r^2,f-X_{n+1}\rangle\subseteq\RR[X,X_{n+1}]$;
  \item Compute intervals $\{[a_i,b_i]\mid i=1,\ldots,s\}$ such that
	  $0\not\in[a_i,b_i]$ for each $i$ and the coordinate $X_{n+1}$
	  of every point in $\bfV_\RR(\bar{I})\subseteq\RR^{n+1}$ lies in some unique $[a_i,b_i]$;
  \item Let $m=\min\{a_i, i=1,\ldots,s\}$ and $M=\max\{b_i,i=1,\ldots,s\}$;
  \item If $m>0$, return ``local minimizer''; if $M<0$, return
	  ``local maximizer''; if $m<0<M$, return ``saddle point''.
\end{enumerate}}
\end{algorithm}
%\hrule
\end{framed}

\begin{theorem}\label{}
%	Suppose that the inputs $\bfg$ and $R$
%	in Algorithm \ref{al::main} satisfies that
%$\langle\bfg\rangle$ is radical, equidimensional one and
%$\Gamma_\RR(f)\cap\mathsf{B}_R =\bfV_\RR(\bfg)\cap\mathsf{B}_R$.
%Then,
Algorithm \ref{al::main} runs successfully and is correct.
\end{theorem}
\begin{proof}
%	By Proposition \ref{prop::cs}, $I_{\Delta_\bfg,Y}\neq\emptyset$ and
%	the lower bound $r>0$ exists. Since
%$\Gamma_\RR(f)\cap\mathsf{B}_R =\bfV_\RR(\bfg)\cap\mathsf{B}_R$,
%By Lemma \ref{lem::cs}, $\mathscr{R}_I>0$ exists.
%By Theorem \ref{th::fr} and Corollary \ref{cor::fr},
%$R_1$ is a faithful radius.
%According to Proposition \ref{prop::r},
	By Proposition \ref{prop::r} and Algorithm \ref{al::fr},
$f^-_{r}$ and $f^+_{r}$ respectively equal the minimal and maximal
coordinates $X_{n+1}$ of the points in $\bfV_\RR(\bar{I})$.
Since $\bar{I}$ is zero-dimensional by Proposition \ref{prop::cs} and
$f^-_r, f^+_r$ are nonzero by Theorem \ref{th::con}, the isolation
intervals $[a_i,b_i]$'s in step 5 can be obtained. Again, by Theorem \ref{th::con},
the outputs of Algorithm is correct.
\end{proof}

\paragraph{\bf Example \ref{ex::ex1} continued}
{\itshape We have shown that any
$0<r<R \approx 0.778<1$ is a faithful radius of $\bfz$.
We set $r=\frac{7}{18}$ in Step 2 of Algorithm \ref{al::main}.
In Step 4, by the command {\sf Isolate} in Maple which uses RUR method
\cite{Rouillier1999} for zero-dimensional system,
we obtain that $m=\frac{76810939241945}{562949953421312}$
and $M=\frac{437849963772149}{1125899906842624}$ in Step 5.
Therefore, we can claim that $\bfz$ is a local minimizer of $f$ by Step 6
of Algorithm \ref{al::main}. }
\vskip 8pt

%In the following, we assume that
%Remark that when $f\in\RR[X]$ is square-free, $\langle f\rangle\subseteq\RR[X]$ is radical.

%$\bfg:=\langle g_1,\ldots,g_s\rangle\subseteq\RR[X]$ such that
%Then, $\dim I_g=1$ in
%generic coordinates by Theorem \ref{th::dim}. Denote $\Delta$ as the union of
%$\{g_1,\ldots,g_s\}$ and the set of all
%$n$-minors of the Jacobian matrix of $\{g_1,\ldots,g_s,\Vert X\Vert_2^2\}$.
%In the following, we assume the coordinates are generic.
%Recall the definitions $(\ref{eq::fbar})$.
%\begin{cor}
%For any isolation radius $R$,
%%with $\crit_\RR(f)\cap\mathsf{B}_R=\{\bfz\}$,
%we have
%\[
%	f_R^{-}=\min\{f(x)\mid x\in \bfV_\RR(\mathcal{G})\cap \mathsf{S}_R\}\quad
%	\text{and}\quad
%	f_R^{+}=\max\{f(x)\mid x\in \bfV_\RR(\mathcal{G})\cap \mathsf{S}_R\}.
%\]
%\end{cor}
%
%
%Recall the definitions in $(\ref{eq::g})$.
%\begin{theorem}
%Up to a generic linear change of coordinates, $I_{\mathcal{G},Y}\neq\emptyset$.
%If $0<r_1<\cdots<r_s$ be the real zeros of $I_{\mathcal{G},Y}$, then
%for any isolation radius $R<r_1$,
%%with $\crit_\RR(f)\cap\mathsf{B}_R=\{\bfz\}$,
%$R$ is a faithful radius and the polynomial system
%$\bfV_\CC(\mathcal{G})\cap\mathsf{S}_{R,\CC}$
%%\begin{equation}\label{eq::ss}
%%	\bfV_\CC(\mathcal{G})\cap
%%	\left\{x\in\CC^n\mid\sum_{i=1}^n x_i^2=R\right\}
%%\end{equation}
%is zero-dimensional.
%\end{theorem}
%\begin{proof}
%	By Corollary \ref{cor::dim1}, Proposition \ref{prop::cs} and Theorem \ref{th::R},
%$I_{\mathcal{G},Y}\neq\emptyset$ and $R$ is faithful radius.
%\end{proof}

%\subsection{}

\begin{example}
	Consider the following polynomial $($cf. \cite{Nie2015,ADPO}$)$
\[
	\begin{aligned}
		f(X_1,X_2,X_3)=&47X_1^5+5X_1X_2^4+33X_3^5-95X_1^4-47X_1X_3^3+51X_2^2X_3^2-92X_1X_3^2
		-70X_2^2X_3+21X_2^2.
	\end{aligned}
\]
It can be checked that $\bfz$ is a degenerate critical point and moreover
the set $\crit_\CC(f)$ is zero-dimensional. Hence, $\bfz$ is isolated real critical point.
Using the command {\sf Isolate} in Maple, we get an isolation
radius $R_{\text{\sf iso}}=\frac{70375577207295}{140737488355328}\approx 0.50$.
Running Algorithm \ref{al::fr} with inputs $f$ and $R_{\text{\sf iso}}$,
we get the output $R=\frac{459690419250099}{4503599627370496}\approx 0.102$.
Setting $r=\frac{2}{39}$ in Step 2 of Algorithm \ref{al::main}, we obtain
$m=-\frac{428092208351331}{2251799813685248}$ and $M=\frac{13589527442797}{70368744177664}$.
Thus, $\bfz$ is a saddle point of $f$. In fact, it can be certified by
letting $X_1=X_2=0$ in $f$.
\end{example}

\begin{example}
	Consider the polynomial $f(X_1,X_2,X_3)=X_1^2+X_2^4+X_3^4-4X_1X_2X_3$ with
	$\bfz$ as a degenerate critical point.  It is shown in \cite{Cushing1975} that
	the method proposed therein fails to test the type of $\bfz$.
	For any $\varepsilon>0$, we have $f(\varepsilon^2,\varepsilon,\varepsilon)<0$ and
	$f(-\varepsilon,\varepsilon,\varepsilon)>0$.
	Letting $\varepsilon\rightarrow 0$, we get that $\bfz$ is a saddle point of $f$.
	Since the set $\crit_\CC(f)=\{\bfz\}$, we set an isolation radius $R_{\text{\sf iso}}=1$.
Running Algorithm \ref{al::fr} with inputs $f$ and $R_{\text{\sf iso}}$,
we get the output $R=\frac{1}{2}$.
Setting $r=\frac{1}{4}$ in Step 2 of Algorithm \ref{al::main}, we obtain
$m=-\frac{90700979328567}{9007199254740992}$ and $M=\frac{5629499534213}{35184372088832}$.
Thus, we can detect that $\bfz$ is a saddle point of $f$.
\end{example}

To conclude this section,
we would like to point out that the cost of running Algorithms \ref{al::fr} and
\ref{al::main} can be reduced if some factor decomposition of $f$ is available.
The following propositions show that the problem of classifying the isolated
real critical point $\bfz$ of $f$ reduces to the case when
%we only need to consider polynomials
$f$ is square-free and each of its factors vanishes at $\bfz$.
\begin{prop}\label{prop::sf}
Suppose $f(X)=g(X)h(X)^2$ where $g(X),h(X)\in\RR[X].$
Then,
\begin{enumerate}[{Case} 1.]
\item $g(\bfz)\neq 0$. If $g_0>0$, then $\bfz$ is a local minimizer of $f(X)$;
otherwise, $\bfz$ is a local maximizer. Here, $g_0$ denotes the constant term of
$g(X)$.
\item $g(\bfz)=0$. If $\nabla g(\bfz)\neq\bfz$, then $\bfz$ is a saddle point of $f(X)$;
otherwise, $\bfz$ is of the same type as a common critical point of $f(X)$ and $g(X)$.
\end{enumerate}
\end{prop}
\begin{proof}
If $g(\mathbf{0}) \ne 0$ then there exists an open neighbourhood $U \subset \mathbb{R}^n$ of $\mathbf{0}$ such that either $g > 0$ or $g < 0$ on $U.$
This implies easily that $\mathbf{0}$ is a local minimizer or maximizer of $f.$

Assume that $g(\mathbf{0}) =  0$ and $\nabla g(\mathbf{0}) \ne \mathbf{0}.$
We have that $\mathbf{0}$ is not a local extremal point of $g.$
Consequently, for any open neighbourhood $U \subset \mathbb{R}^n$ of $\mathbf{0}$,
there exist points $u, v \in U$ such that $g(u) < 0 < g(v).$
On the other hand, the algebraic set $h^{-1}(0)$ has dimension $< n.$
By the continuity of $g$, we may assume that $u, v \not \in h^{-1}(0).$
Therefore,  $f(u) < 0 < f(v),$ and so $\mathbf{0}$ is a saddle point of $f.$

Finally, suppose that $g(\mathbf{0}) = 0$ and $\nabla g(\mathbf{0}) = \mathbf{0}.$
Because for any $x \in \mathbb{R}^n$ with $f(x) \neq 0$,
$f(x)$ and $g(x)$ have the same sign,
$\mathbf{0}$ is of the same type as a common critical point of $f$ and $g.$
\end{proof}

\begin{prop}
Suppose that $\bfz$ is a critical point of $f\in\RR[X]$ and
$f=f_1\cdot f_2$ where $f_1,f_2\in\RR[X]$. If $f_1(\bfz)\neq 0$,
then
\begin{enumerate}[\upshape (i)]
\item $\bfz$ is a saddle point of $f$ if and only if $\bfz$ is
a saddle point of $f_2$;
\item If $f_1(\bfz)>0$ $(f_1(\bfz)<0$, resp.$)$,
	then $\bfz$ is a minimizer of $f$ if
and only if $\bfz$ is a minimizer (maximizer, resp.) of $f_2$.
\end{enumerate}
\end{prop}
\begin{proof}
Since $f_2(\bfz)=0$ and $\nabla f_2(\bfz)=\bfz$, the conclusion is clear.	
\end{proof}

\section{Conclusions}\label{sec::cons}
We proposed a computable and symbolic method to determine the type of
a given isolated real critical point, which is degenerate, of a multivariate
polynomial function. Given an isolation radius of the critical point,
the tangency variety of the polynomial function at the
critical point is used to define and compute its faithful radius.
Elimination ideals and root isolation of univariate polynomials are computed in finding
a faithful radius.
Once a faithful radius of the critical point is known, its type can be determined
by isolating the real roots of a zero-dimensional polynomial system.

\end{document}